\documentclass{amsart}
\usepackage{amssymb, enumerate, xspace, graphicx, url}
\usepackage[latin1]{inputenc}
\usepackage[all]{xy}
\SelectTips{cm}{}

\numberwithin{equation}{section}

\numberwithin{subsection}{section}

\allowdisplaybreaks[1]

\newenvironment{enumeratea}
{\begin{enumerate}[\upshape (a)]}
{\end{enumerate}}

\newenvironment{enumerate1}
{\begin{enumerate}[\upshape (1)]}
{\end{enumerate}}

\newenvironment{enumeratei}
{\begin{enumerate}[\upshape (i)]}
{\end{enumerate}}

\newtheorem*{namedtheorem}{\theoremname}
\newcommand{\theoremname}{testing}

\newtheorem{theorem}{Theorem}[section]
\newtheorem{proposition}[theorem]{Proposition}
\newtheorem{proposition-definition}[theorem]
{Proposition-Definition}
\newtheorem{corollary}[theorem]{Corollary}
\newtheorem{lemma}[theorem]{Lemma}

\theoremstyle{definition}
\newtheorem{definition}[theorem]{Definition}

\newtheorem{example}[theorem]{Example}
\newtheorem{examples}[theorem]{Examples}
\newtheorem{remark}[theorem]{Remark}

\theoremstyle{remark}

 \newcommand\cB{\mathcal{B}}
\newcommand\cC{\mathcal{C}} \newcommand\cD{\mathcal{D}}
\newcommand\cE{\mathcal{E}}

\newcommand\cM{\mathcal{M}} \newcommand\cN{\mathcal{N}}
\newcommand\cO{\mathcal{O}}

\renewcommand\AA{\mathbb{A}}

\newcommand\GG{\mathbb{G}}

 \newcommand\NN{\mathbb{N}}

 \newcommand\ZZ{\mathbb{Z}}

\newcommand\rmm{\mathrm{m}} 
 \newcommand\rmp{\mathrm{p}}

\newcommand\arr{\ifinner\to\else\longrightarrow\fi}

\newcommand\arrto{\ifinner\mapsto\else\longmapsto\fi}
\newcommand\larr{\longrightarrow}

\newcommand{\xarr}{\xrightarrow}

\newcommand\op{^{\mathrm{op}}}

\newcommand\eqdef{\overset{\mathrm{\scriptscriptstyle def}} =}

\renewcommand\th{^\text{th}}

\def\displaytimes_#1{\mathrel{\mathop{\times}\limits_{#1}}}

\def\displayotimes_#1{\mathrel{\mathop{\bigotimes}\limits_{#1}}}

\renewcommand\hom{\operatorname{Hom}}

\newcommand\aut{\operatorname{Aut}}

\newcommand\pic{\operatorname{Pic}}

\newcommand\spec{\operatorname{Spec}}

\newcommand\id{\mathrm{id}}

\newcommand{\cat}[1]{(\mathrm{#1})}

\newcommand\dash{\nobreakdash-\hspace{0pt}}

\newdir{ >}{{}*!/-5pt/@{>}}

\newcommand\doublelong[2]{\mathbin{\xymatrix{{}\ar@<3pt>[r]^{#1}
\ar@<-3pt>[r]_{#2}&}}}

\newcommand{\underhom}
{\mathop{\underline{\mathrm{Hom}}}\nolimits}

\newcommand{\underspec}
{\mathop{\underline{\mathrm{Spec}}}\nolimits}

\newlength{\ignora}

\newcounter{stepcount}

\newcommand{\catsch}[1]{(\mathrm{Sch}/#1)}

\newcommand{\qc}{quasi-coherent\xspace}

\newcommand{\df}{Deligne--Faltings\xspace}
\newcommand{\pdf}{pre-Deligne--Faltings\xspace}

\newcommand{\sm}{symmetric monoidal\xspace}

\newcommand{\et}{_{\text{\'et}}}

\newcommand{\gp}{^{\mathrm{gp}}}

\newcommand{\gm}{\GG_{\rmm}}

\newcommand{\catset}{(\mathrm{Set})}
\newcommand{\catmon}{(\mathrm{ComMon})}
\newcommand{\catsymmon}{(\mathrm{SymMonCat})}
\newcommand{\one}{\mathbf{1}}

\renewcommand{\pic}{\operatorname{\mathfrak{Pic}}}
\renewcommand{\div}{\operatorname{\mathfrak{Div}}}
\newcommand{\qcoh}{\operatorname{\mathfrak{QCoh}}}

\newcommand\radice[2]{\sqrt[\uproot{2}#1]{#2}}

\newcommand{\sh}{^{\mathrm{sh}}}
\newcommand{\sep}{^{\mathrm{sep}}}
\newcommand{\pre}{^{\mathrm{pre}}}
\newcommand{\wt}{^{\mathrm{wt}}}
\newcommand{\dwt}{^{\mathrm{dwt}}}

\newcommand{\ass}{^{\mathrm{a}}}
\newcommand{\pass}{^{\mathrm{pa}}}

\newcommand{\sz}[1]{\spec\ZZ[#1]}
\newcommand{\szp}[1]{[\spec\ZZ[#1]/\widehat{#1}]}

\DeclareMathOperator{\obj}{Obj}

\DeclareMathOperator{\cl}{cl}
\DeclareMathOperator{\Hom}{Hom}
\DeclareMathOperator{\bHom}{\mathbf{Hom}}

\begin{document}

\bibliographystyle{amsalpha}

\title{Parabolic sheaves on logarithmic schemes}

\author[Borne]{Niels Borne}

\author[Vistoli]{Angelo Vistoli}

\address[Borne]{Laboratoire Paul Painlevé\\
Université de Lille\\
U.M.R. CNRS 8524\\
U.F.R. de Mathématiques\\
59\,655 Villeneuve d'Ascq Cédex\\
France}
\email{Niels.Borne@math.univ-lille1.fr}

\address[Vistoli]{Scuola Normale Superiore\\Piazza dei Cavalieri 7\\
56126 Pisa\\ Italy}
\email{angelo.vistoli@sns.it}

\date{\today}

\maketitle

\begin{abstract}
We show how the natural context for the definition of parabolic sheaves on a scheme is that of logarithmic geometry. The key point is a reformulation of the concept of logarithmic structure in the language of symmetric monoidal categories, which might be of independent interest. Our main result states that parabolic sheaves can be interpreted as quasi-coherent sheaves on certain stacks of roots.
\end{abstract}

\section{Introduction}

The notion of \emph{parabolic bundle} on a curve was introduced by Mehta and Seshadri (see \cite{mehta-seshadri} and \cite{seshadri}), and subsequently generalized to higher dimension by Maruyama and Yokogawa (\cite{maruyama-yokogawa}); this latter definition was later improved by Mochizuki (\cite{mochizuki}), Iyer--Simpson (\cite{iyer-simpson}) and the first author (\cite{borne-para2}). Another important insight is due to Biswas, who connected rational parabolic bundles with bundles on orbifolds (\cite{biswas}). The first author refined Biswas' idea in \cite{borne-para1} and \cite{borne-para2}; in the latter paper he proved that, given $n$ smooth effective divisors $D_{1}$, \dots,~$D_{n}$ intersecting transversally on a normal variety $X$, there is an equivalence between the category of rational parabolic bundles and the limit of the category of vector bundles on the fibered product of $d\th$ root stacks of the $(X, D_{i})$, as $d$ becomes very divisible.

This left two main questions open.

\begin{enumerate1}

\item What about divisors that are not simple normal crossing? If, for example, a normal crossing divisor has singular components, it seems clear that one should use sheaves of weights.

\item What is the correct definition of parabolic coherent sheaf? \cite{maruyama-yokogawa} and \cite{iyer-simpson} contain definitions of torsion-free parabolic coherent sheaves; but the definition of a general coherent sheaf has to be essentially different.

\end{enumerate1}

The key point to solving these problems is the introduction of logarithmic structures. The main purpose of this paper is to give a definition of parabolic quasi-coherent sheaf with fixed rational weights on a logarithmic scheme, and to show the equivalence of category of such sheaves with the category of sheaves on a root stack.

More precisely, suppose that $\rho\colon M\rightarrow \cO_{X}$ is a logarithmic structure on a scheme $X$; denote, as usual, by $\overline{M}$ the quotient sheaf $M/\cO^{*}_{X}$. The denominators are taken in an appropriate sheaf of monoids $B$ containing $\overline{M}$; then we define a category of quasi-coherent parabolic sheaves on a fine logarithmic scheme $(X, M, \rho)$ with weights in $B$.

Also, we define a root stack $X_{B/\overline{M}}$; this is a tame Artin stack over $X$. The idea of the construction is essentially due to Martin Olsson, who defined it in several particular cases, from whom it was easy to extract the general definition ((\cite{MO}, \cite{olsson-log-twisted})). If the logarithmic structure is generated by a single effective Cartier divisor $D \subseteq X$, so that $\overline{M}$ is the constant sheaf $\NN_{D}$ on $D$, and we take $B$ to be $\frac{1}{d}\NN_{D}$, then $X_{B/\overline{M}}$ is the root stack $\radice{d}{(X,D)}$ (see \cite{dan-tom-angelo2008} or \cite{cadman}). Our main result (Theorem~\ref{thm:main}) is that the category of \qc parabolic sheaves on a fine logarithmic scheme $(X, M, \rho)$ with weights in $B$ is equivalent to the category of \qc sheaves on the stack $X_{B/\overline{M}}$. This represents a vast generalization of the correspondence of \cite{borne-para2}.

In order to do this we need to interpret the logarithmic structure $(M, \rho)$ as a \sm functor $\overline{M} \arr \div_{X}$, where $\div_{X\et}$ is the \sm stack over the small étale site $X\et$ of $X$ whose objects are invertible sheaves with sections. We call this a \emph{\df structure}. The fact that a \df structure defines a logarithmic structure is somehow implicit in the original construction of the logarithmic structure associated with a homomorphism of monoids $P \arr \cO(X)$, as in \cite{Kato}; going in the other direction, the construction is contained in Lorenzon's paper \cite{lorenzon}.

We find that this point of view has some advantages, and in this paper we make an effort to develop the theory of \df structures systematically, without referring to known results on logarithmic structures. We are particularly fond of our treatment of charts, in \ref{subsec:charts}, which we find somewhat more transparent than the classical one. The resulting notion of fine logarithmic structure is equivalent to the classical one.

There is much left to do in the direction that we point out. Suppose that $(M, \rho)$ is a saturated logarithmic structure. Then we can associate with it a tower of stacks $X_{d} \eqdef X_{\frac{1}{d}\overline{M}/\overline{M}}$, letting $d$ range over all positive integer. This tower seems to control much of the geometry of the logarithmic scheme $(X, M, \rho)$; for example, the limit of the small étale sites of the $X_{d}$ (appropriately restricted when not in characteristic~$0$) is the Kummer-étale site of $(X, M, \rho)$, and one can use the $X_{d}$ to investigate many questions concerning this site; for example, the K-theory of $(X, M, \rho)$, as defined by Hagihara and Nizio{\l} (see \cite{hagihara} and \cite{niziol-K-log}) is naturally interpreted in this language. In subsequent papers we plan to prove Nori's theorem for logarithmic schemes, in the style of \cite{borne-para2}, define real parabolic sheaves, and connections on them (as was pointed out to us by Arthur Ogus, this is important to study the Riemann--Hilbert correspondence for logarithmic schemes, as in \cite{ogus-logRH}), and in general apply this construction to other foundational questions in the theory of logarithmic schemes.

\subsection*{Description of contents} Section~\ref{sec:preliminary} contains several preliminary notions, mostly known, concerning monoids, sheaves of monoids, \sm categories and fibered symmetric monoidal categories. We also define one of our basic notions, that of \emph{\df object}: a \sm functor from a monoid to a \sm category with trivial kernel.

The main definition, that of \df structure (Definition~\ref{def:df-structure}), is contained in Section~\ref{sec:df-structures}. Here we also show the equivalence of the notion of \df structure with that of quasi-integral logarithmic structure (Theorem~\ref{thm:df<->log}). In (\ref{subsec:direct-inverse}) we define direct and inverse images of a \df structure, without going through the associated logarithmic structure.

Our treatment of charts is contained in (\ref{subsec:charts}); we compare it with Kato's treatment in (\ref{subsec:kato-charts}). Some of the results are strictly related with those in \cite[Section~2]{olsson-ens}; we do not refer to this paper, but prefer to reconstruct the theory independently. Proposition~\ref{prop:coherent<->coherent} implies that our resulting notion of fine structure coincides with the classical one.

Section~\ref{sec:sys-roots} contains the notion of systems of denominators and the definition of root stacks. In Section~\ref{sec:parabolic} we define parabolic sheaves and prove their basic properties. Finally, our main result, giving an equivalence between parabolic sheaves and sheaves on a root stack is in Section~\ref{sec:main-theorem}.

\subsection*{Acknowledgments} We would like to thank Luc Illusie, Arthur Ogus, and Martin Olsson for useful conversations.

\section{Definitions and preliminary results}\label{sec:preliminary}

\subsection{Conventions} The class of objects of a category $\cC$ will be denoted by $\obj\cC$. If $F$ is a presheaf (of sets, monoids, groups, \dots) on a site, we denote by $F\sh$ the associated sheaf.

If $F$, $F'\colon \cC \arr \cD$ and $G\colon \cD \arr \cE$ are functors and $\alpha\colon F \arr F'$ is a natural transformation, we denote by $G \circ\alpha$, or simply $G\alpha$, the natural transformation $GF \arr GF'$ defined by the obvious rule $(G\alpha)_{C} = G(\alpha_{C})$. Analogously, if $F\colon \cC \arr \cD$, $G$, $G'\colon \cD \arr \cE$ are functors and $\alpha\colon G \arr G'$ is a natural transformation, we denote by $\alpha\circ F$ or $\alpha F$ the natural transformation $GF \arr G'F$ defined by $(\alpha F)_{C} = \alpha_{F(C)}$.

\subsection{Monoids} All monoids considered will be commutative; we will use additive notation. We denote by $\catmon$ the category of (commutative) monoids.

If $A$ is a monoid, we denote the associated group by $A\gp$, and by $\iota_{A}\colon A \arr A\gp$ the canonical homomorphism of monoids. Any element of $A\gp$ is of the form $\iota_{A}a - \iota_{A}b$ for some $a$, $b \in A$; furthermore, two elements $a$ and $b$ of $A$ have the same image in $A\gp$ if and only if there exists $c \in A$ such that $a+ c = b + c$. A homomorphism of monoids $f\colon A \arr B$ induces a group homomorphism $f\gp \colon A\gp \arr B\gp$. 

A monoid is \emph{integral} if the cancellation law holds; equivalently, a monoid is integral if $\iota_{A}\colon A \arr A\gp$ is injective. A monoid $A$ is \emph{torsion-free} if it is integral and $A\gp$ is torsion-free. If $f\colon A \arr B$ is an injective homomorphism of monoids and $B$ is integral, then $f\gp \colon A\gp \arr B\gp$ is also injective.

A monoid is \emph{sharp} if the only invertible element is the identity. Notice that a sharp monoid has no non-zero element of finite order; however, the associated group $A\gp$ is not necessarily torsion-free.

The \emph{kernel} of a homomorphism of monoids $f\colon A \arr B$ is $f^{-1}(0) \subseteq A$. In contrast with the case of groups, the kernel of $f$ may be trivial without $f$ being injective (for example, look at the homomorphism $\NN^{2} \arr \NN$ defined by $(x,y) \arrto x+y$). An arbitrary submonoid $S \subseteq A$ is not necessarily a kernel. The following condition is necessary and sufficient for $S$ to be a kernel: if $a \in A$, $s \in S$ and $a+s \in S$, then $a \in S$.

If $S$ is a submonoid of $A$, we denote by $A/S$ the cokernel of the inclusion $S \subseteq A$. This is the quotient of $A$ by the equivalence relation $\sim_{S}$ defined by $a \sim b$ when there exist $s \in S$ and $t \in S$ such that $a + s = b + t$. The kernel of the projection $A \arr A/S$ is the set of $a \in A$ such that there exists $s \in S$ with $a + s\in S$. This is the smallest kernel that contains $S$, and we call it the \emph{kernel closure} of $S$.

A homomorphism of monoids $A \arr B$ is called a \emph{cokernel} if it is the cokernel of a homomorphism $C \subseteq A$. Any cokernel is surjective, but not every surjective homomorphism is a cokernel. A necessary and sufficient condition for $f\colon A \arr B$ to be a cokernel is that if $K$ is the kernel of $f$, the induced homomorphism $A/K \arr B$ is an isomorphism.

\subsection{Sheaves of monoids} Many of the notions above extend to sheaves and presheaves of monoids on a site $\cC$. If $A$ is such a sheaf, we define $A\gp$ as the sheafification of the presheaf sending $U \in \obj\cC$ into $A(U)\gp$. The obvious homomorphism of sheaves of monoids $\iota_{A}\colon A \arr A\gp$ is universal among homomorphism of sheaves of monoids from $A$ to a sheaf of groups.

A presheaf of monoids is called \emph{integral} if each $A(U)$ is integral. If $A$ is integral, so is the associated sheaf $A\sh$. It is \emph{sharp} if each $A(U)$ is sharp. 

If $K$ is a sub-presheaf of monoids of a presheaf of monoids $A$, we can define the presheaf quotient $A/K$ by the rule $(A/K)(U) = A(U)/K(U)$. It is the cokernel in the category of presheaves of monoids of the inclusion $K \subseteq A$. In general, if $C \arr A$ is a homomorphism of presheaves, its cokernel is $A/K$, where $K$ is the image presheaf in $A$.

If we substitute presheaves with sheaves, the quotient $A/K$ is the sheafification of the presheaf quotient; all cokernels in the category of sheaves of monoids are of this type.

\subsection{Symmetric monoidal categories} 

Our treatment of logarithmic structures is centered around the notion of \emph{\sm category}. We freely use the notation and the results of \cite[ch. VII and XI]{McLan}, which will be our main reference. This concept was introduced in \cite[II, Definition 2.1]{DMOS}, under the name \emph{tensor category}.

Let $\cM$ a \sm category. We denote the operation (the ``tensor product'') by $\otimes\colon \cM \times \cM \arr \cM$, its action on objects and arrows by $(x, y) \arrto x \otimes y$, the neutral element of $\cM$ by $\one$, the associativity isomorphisms $x \otimes (y \otimes z) \simeq (x \otimes y) \otimes z$ by $\alpha$, or $\alpha_{x,y,z}$, the isomorphism $\one \otimes x \simeq x$ by $\lambda$ or $\lambda_{x}$, the isomorphism $x \otimes y \simeq y \otimes x$ by $\sigma$, or $\sigma_{x,y}$. Occasionally we will use the subscript ${}_{\cM}$ (as in $\otimes_{\cM}$, $\one_{\cM}$, and so on) to distinguish among such objects relative to different \sm categories.

If $\cM$ and $\cN$ are \sm categories, a \emph{\sm functor} $F\colon \cM \arr \cN$ will be is a strong braided monoidal functor $\cM \arr \cN$ (\cite[ch.~IX, \S~2]{McLan}). All natural transformations between \sm functors will be assumed to be monoidal.

We denote by $\catsymmon$ the 2-category of symmetric monoidal categories. The objects are small symmetric monoidal categories, the 1-arrows are symmetric monoids functors, and the 2-arrows are monoidal natural transformations.

If $F\colon \cM \arr \cN$ is a \sm functor, which is an equivalence, when viewed as a functor of plain categories, then any quasi-inverse $G\colon \cN \arr \cM$ has a unique structure of a \sm functor, such that the given isomorphisms $FG \simeq \id_{\cN}$ and $GF \simeq \id_{\cM}$ are monoidal isomorphisms.

Any monoid $A$ will be considered as a discrete \sm category: the arrows are all identities, while the tensor product is the operation in $A$.

For the convenience of the reader, we make the notion of a \sm functor explicit in the case that we will use the most.

\begin{definition}\label{def:symmetric-monoidal-functor}
Let $A$ be a monoid, $\cM$ a \sm category. A \emph{\sm functor} $L\colon A \arr \cM$ consists of the following data.

\begin{enumeratea}

\item A function $L\colon A \arr \obj\cM$.

\item An isomorphism $\epsilon^{L}\colon \one \simeq L(0)$ in $\cM$.

\item For each $a$ and $b\in A$, an isomorphism $\mu^{L}_{a,b}\colon  L(a) \otimes L(b)\simeq L(a+b)$ in $\cM$.

\end{enumeratea}

We require that for any $a$, $b$, $c \in A$, the diagrams
   \[
   \xymatrix@C = 30pt{
   L(a)\otimes \bigl(L(b)\otimes L(c)\bigr)
      \ar[r]^-{\id \otimes \mu^{L}}\ar[dd]^{\alpha}
   &L(a)\otimes L(b+c)\ar[rd]^{\mu^{L}}& \\
   && L(a+b+c)\,,\\
   \bigl(L(a)\otimes L(b)\bigr)\otimes L(c)
      \ar[r]^-{\mu^{L} \otimes \id}
   &L(a+b) \otimes L(c)\ar[ru]^{\mu^{L}}&
   }
   \]
   
   \[
   \xymatrix{
   L(a+b) \ar@{=}[d] \ar[r]^-{\mu^{L}}
   & L(a) \otimes L(b) \ar[d]^{\sigma}\\
   L(b+a)\ar[r]^-{\mu^{L}}&
   L(b)\otimes L(a)
   }
   \]
and
   \[
   \xymatrix{
   L(a)\ar[d]^{=}
   &  \one\otimes L(a)\ar[l]_{\lambda} \ar[d]^{\epsilon^{L}\otimes \id}\\
   L(0+a) & L(0)\otimes L(a)\ar[l]_-{\mu^{L}}
   }
   \]
be commutative.

If $L\colon A \arr \cM$ and $M\colon A \arr \cM$ are \sm functors, a \emph{morphism} $\Phi\colon L \arr M$ is collection of natural transformations $\Phi_{a}\colon L(a) \arr M(a)$, one for each $a \in A$, such that for any $a$, $b \in A$, the diagram
   \[
   \xymatrix{
   L(a)\otimes L(b) \ar[r]^-{\mu^{L}}\ar[d]^{\Phi_{a} \otimes \Phi_{b}}
   & L(a+b) \ar[d]^{\Phi_{a+b}}\\
   M(a)\otimes M(b) \ar[r]^-{\mu^{M}}
   & M(a+b)
   }
   \]
commutes.
\end{definition}

If $f\colon A \arr B$ is a homomorphism of monoids, and $L\colon B\arr \cM$ is a \sm functor, the composite $L \circ f\colon A \arr \obj\cM$ has an obvious structure of a \sm functor. We will use this fact, together with some evident properties, without comments.

\begin{definition}
Let $\cM$ be a \sm category $\cM$. An \emph{$\cM$-valued \df object} is a pair $(A,L)$, where $A$ is a monoid and $L\colon A \arr \cM$ is a \sm functor $L\colon A \arr \cM$.
\end{definition}

There is a category of $\cM$-valued \df objects. An arrow from $(A, L)$ to $(B, M)$ (which we call a \emph{morphism of \df objects}) is a pair $(\phi, \Phi)$, where $\phi\colon A \arr B$ is a homomorphism of monoids and $\Phi\colon L \arr M \circ \phi$ is a monoidal natural transformation. The composition is defined in the obvious way.

\begin{definition}
Let $\cM$ be a \sm category, $L\colon A \arr \cM$ be a \df object. The \emph{kernel} $\ker L$ is the set of elements $a \in A$ such that $L(a)$ is isomorphic to the neutral element $\one$.
\end{definition}

One checks immediately that $\ker L$ is a sub-monoid of $A$.

\begin{proposition}
\label{prop:factor-predf}
Let $L\colon A \arr \cM$ be a \df object. Assume that the monoid of endomorphisms of the neutral element $\one$ of $\cM$ is trivial.

Then there exists a cokernel $\pi\colon A \arr \overline{A}$ and a \sm functor $\overline{L}\colon \overline{A} \arr \cM$ with trivial kernel, with an isomorphism of \sm functors $\overline{L} \circ \pi \simeq L$. Furthermore the \df object $(\overline{A}, \overline{L})$ is unique, up to a unique isomorphism, and is universal among morphism from $(A,L)$ to a \df object $(B,M)$ with trivial kernel.
\end{proposition}

\begin{proof}
It is clear that if $(\overline{A}$, $\overline{L})$ and the isomorphism exist, the kernel of $\pi$ must be the kernel $K$ of $L$; therefore there is a unique isomorphism of $\overline{A}$ with $A/K$. So it is enough to show that there exists a unique factorization
   \[
   A \stackrel{\pi}{\larr} A/K \stackrel{\overline{L}}{\larr} \cM,
   \]
up to a unique isomorphism, and that it has the required universal property. 

Notice that if $k \in K$, the isomorphism $\one \simeq L(k)$, which exists by hypothesis, must be unique, because $\aut(\one)$ is trivial. Hence for any $a$ in $A$ we get canonical isomorphisms
   \[
   L(a) \stackrel{\lambda}\simeq \one \otimes L(a)
   \simeq L(k)\otimes L(a) \stackrel{\mu^{L}}\simeq L(k+a).
   \]
Denote by $\nu_{k,a}$ the resulting canonical isomorphism $L(a) \simeq L(k+a)$.  This isomorphism is easily shown to have the property that
   \[
   \nu_{l+k,a}=\nu_{l,a+k} \circ \nu_{k,a}\colon L(a) \simeq L(l+k+a)
   \]
for any $k$, $l \in K$.

Now, suppose that $a$ and $b \in A$ have the same image in $A/K$. There exist $k$, $l\in K$ such that $k+a = l+b$; so we have an isomorphism $\tau_{a,b}\colon L(a) \simeq L(b)$ defined as the composite
   \[
   L(a) \stackrel{\nu_{k,a}}\simeq L(k+a) = L(l+b)
   \stackrel{\nu_{l,b}^{-1}}\simeq L(b).
   \]
It is easy to check that $\tau_{a,b}$ is independent of the choice of $k$ and $l$. If $a$, $b$ and $c$ have the same image in $A/K$, then we can find $k$, $l$ and $m \in K$ such that $k + a = l + b = m + c$, and then
   \[
   \tau_{a,c} = \tau_{b,c} \circ \tau_{a,b}\colon L(a) \simeq L(b).
   \]

Now consider the category $A^{K}$, whose object are the elements of $A$, and in which, given two elements $a$, $b \in A$, there exists exactly one arrow $a \arr b$ if $a$ and $b$ have the same image in $A/K$, and none otherwise. The category $A^{K}$ is a strict \sm category, with the tensor product given by the operation in $A$. The projection $\pi\colon A \arr A/K$ factors through $A^{K}$, and the projection $A^{K} \arr A/K$ is an equivalence of monoidal categories. Hence it is enough to produce a factorization 
   \[
   A \arr A^{K} \stackrel{\widehat{L}}{\larr} \cM,
   \]
and then the desired functor $\overline{L}\colon A/K \arr \cM$ will be obtained by composing $\widehat{L}$ with a quasi-inverse $A/K \arr A^{K}$. This factorization is obtained by defining $\widehat{L}$ to be the same function as $L$ on the elements of $A$; if $a \arr b$ is an arrow in $A^{K}$, then we take as its image in $\cM$ the isomorphism $\tau_{a,b}\colon L(a) \simeq L(b)$. We leave it to the reader to check that this functor is monoidal, and gives the desired factorization.

We have left to check that $(\overline{A}, \overline{L})$ has the desired universal property. Suppose that $(\phi,\Phi)\colon (A, L) \arr (B,M)$ is a morphism of \df objects. Let $K'$ be the kernel of $M$: clearly $\phi$ sends $K$ to $K'$, thus there is a natural commutative diagram:

 \[
   \xymatrix{
   (A,L) \ar[d] \ar[r]   &  (A^{K},\widehat{L})\ar[d] \\
    (B,M)  \ar[r]   &  (B^{K'},\widehat{M})  }
   \]
If the kernel $K'$ of $(B,M)$ is trivial, the bottom morphism is an isomorphism, and this shows existence in the universal property. We leave it to the reader to check uniqueness.
\end{proof}

The following two examples play a key role in this paper.

\begin{examples} Let $X$ be a scheme.

\begin{enumeratea}

\item We denote by $\div X$ the groupoid of line bundles with sections. We consider $\div X$ as a category of ``generalized effective Cartier divisors'': effective Cartier divisors on $X$ form a monoid, which is, however, not functorial in $X$, since one can't pull back Cartier divisors along arbitrary maps. Line bundles with sections don't have this problem; there is a price to pay, however, which is to have to deal with a \sm category instead of a monoid.

The objects of $\div X$ are pairs $(L, s)$, where $L$ is an invertible sheaf on $X$ and $s \in L(X)$. An arrow from $(L, s)$ to $(L', s')$ is an isomorphism of $\cO_{X}$-modules from $L$ to $L'$ carrying $s$ into $s'$. The category $\div X$ also has a \sm structure given by tensor product, defined as $(L, s) \otimes (L', s') \eqdef (L \otimes L', s \otimes s')$. The neutral element is $(\cO_{X},1)$.

Notice that $\div X$ has the property that the monoid of endomorphisms of the neutral element $(\cO_{X},1)$ is trivial.

\item We denote by $\pic X$ the category of invertible sheaves on $X$, with the monoidal structure given by tensor product. We notice that, in contrast with standard usage, and with the example above, the arrows in $\pic X$ will be arbitrary homomorphisms of $\cO_{X}$-modules, and not only isomorphisms. Thus, $\pic X$ is not a groupoid. Tensor product makes $\pic X$ into a \sm category, with neutral element $\cO_{X}$. 

\item The category of invertible sheaves on $X$, in which the only arrows are the isomorphisms, will be denoted by $\cB\gm(X)$.

\end{enumeratea}

\end{examples}

There is a natural strict \sm functor $\div X \arr \pic X$ sending $(L, s)$ to $L$.

If $A$ is a monoid and $\cM$ is a \sm category, we denote by $\hom(A, \cM)$ the category of \sm functors $A \arr \cM$. Given a homomorphism of monoids $f\colon A \arr B$, there is an induced functor $f^{*}\colon \hom(B, \cM) \arr \hom(A, \cM)$ sending $L\colon B \arr \cM$ into the composite $L \circ f\colon A \arr \cM$.

\subsection{Monoidal fibered categories}

Here we will freely use the language of fibered categories, for which we refer to \cite[Chapter~3]{fga-explained}. 

\begin{definition}
Let $\cC$ be a category. A \emph{\sm fibered category} $\cM \arr \cC$ is a fibered category, together with a cartesian functor
   \[
   \otimes = \otimes_{\cM}\colon \cM \times_{\cC} \cM \arr \cM,
   \]
a section $\one_{\cM}\colon \cC \arr \cM$, and base-preserving natural isomorphisms
   \[
   \alpha\colon 
      \otimes\circ(\id_{\cM} \times \otimes) \simeq 
\otimes\circ(\otimes \times \id_{\cM})
   \quad\text{of functors}\quad
   \cM\times_{\cC}\cM \times_{\cC} \cM \arr \cM,
   \]
   \[
   \lambda\colon 
   \otimes\circ(\one_{\cM} \times \id_{\cM}) \simeq \id_{\cM}
   \quad\text{of functors}\quad
   \cM \simeq \cC\times_{\cC}\cM \arr \cM,
   \]
and
    \[
   \sigma\colon 
   \otimes \simeq \otimes\circ \Sigma_{\cM}
   \quad\text{of functors}\quad
   \cM\times_{\cC}\cM \arr \cM,
   \]
where by $\Sigma_{\cM}\colon \cM\times_{\cC}\cM \arr \cM\times_{\cC}\cM$ we mean the functor exchanging the two terms, such that for any object $U$ of $\cC$ the restrictions of $\otimes$ and of the natural transformations above yield a structure of \sm category on $\cM(U)$.

If $\cM \arr \cC$ and $\cN \arr \cC$ are \sm fibered categories, a \sm functor $F\colon \cM \arr \cN$ is a cartesian functor, together with an isomorphism
   \[
   \mu^{L}\colon 
   \otimes_{\cN} \circ (F \times F) \simeq F \circ \otimes_{\cM}
   \quad\text{of functors}\quad
   \cM \times_{\cC} \cM \arr \cN
   \]
and
   \[
   \epsilon^{F}\colon 
    \one_{\cN}\simeq F\circ \one_{\cM}    
    \quad\text{of functors}\quad
   \cC \arr \cN,
   \]
such that the restrictions of these data to each $\cM(U)$ and $\cN(U)$ gives $F_{U}\colon \cM(U) \arr \cN(U)$ the structure of a \sm functor.

Morphisms of \sm functors are base-preserving natural transformation, whose restriction to each fiber is monoidal.
\end{definition}

If $\cM \arr \cC$ is a \sm fibered category and we choose a cleavage for it, we obtain a pseudo-functor (i.e., a lax 2-functor) from $\cC\op$ the 2-category $\catsymmon$ of \sm categories. A different choice of a cleavage yields a canonically isomorphic pseudo-functor. 

Conversely, given a pseudo-functor $\cC\op \arr \catsymmon$, the usual construction yields a \sm fibered category over $\cC$ with a cleavage.

In particular, if $A\colon \cC\op \arr \catmon$ is a presheaf of monoids on a category $\cC$, we consider the associated fibered category $(\cC/A) \arr \cC$. The objects of $(\cC/A)$ are pairs $(U, a)$, where $U$ is an object of $\cC$ and $a \in A(U \to X)$. The arrows from $(U, a)$ to $(V, b)$ are arrows $f: U \arr V$ such that $f^{*}b = a$. Because of the customary identification of categories fibered in sets on $\cC$ and functors $\cC\op \arr \catset$, we will usually write this simply as $A \arr \cC$. Such a category has a canonical structure of strict \sm fibered category.

\begin{definition}
Let $\cC$ be a site. A \emph{\sm stack over $\cC$} is a \sm fibered category over $\cC$ that is a stack.
\end{definition}

\begin{examples} Let $X$ be a scheme; denote by $\catsch{X}$ the category of schemes over $X$.
\begin{enumeratea}

\item The \sm stack $\div_{X} \arr \catsch X$ is the category associated with the pseudo-functor that sends each $U \arr X$ into the category $\div U$.

\item Analogously, one defines the \sm stack $\pic_{X}$ whose fiber over $U \arr X$ is $\pic U$.

\end{enumeratea}
\end{examples}

\begin{remark}
The stack $\div_{X}$ can be described using the language of algebraic stacks as the quotient $[\AA_{X}^{1}/\GG_{\rmm, X}]$.

The stack $\pic_{X}$ is not an algebraic stack, because it is not a stack in groupoids. The underlying stack in groupoids (obtained by deleting all the arrows that are not cartesian) is the usual Picard stack of $X$, and can be described as the classifying stack $\cB_{X}\gm$ of the group scheme $\GG_{\rmm, X}$, or, in other words, as the stack quotient $[X/\GG_{\rmm, X}]$ for the trivial action of $\gm$ on $X$.
\end{remark}

We will need the following extension result. Let $\cC$ be a site, and let $A\colon \cC\op \arr \catmon$ be a presheaf of monoids on $\cC$. Denote by $A\sh$ the associated sheaf of monoids, $s_{A}\colon A \arr A\sh$ the canonical homomorphism.

\begin{proposition}\label{prop:sheafification}
Let $\cM$ be a \sm  stack over $\cC$, $L\colon  A \arr \cM$ a \sm functor. Then there exists a \sm functor $L\sh\colon A\sh \arr \cM$ and an isomorphism of the composite $s_{A} \circ L\sh\colon  A \arr \cM$ with $L$. Furthermore, the pair $(A\sh, L\sh)$ has the following universal property: any morphism  of \df objects $(A, L) \arr (B,M)$, where $B$ is a sheaf, factors uniquely through $(A\sh, L\sh)$. If $L$ has trivial kernel, so has $L\sh$.
\end{proposition}

This could be considered as obvious, as it says that the sheafification of a presheaf coincides with its stackification. However, we don't know a reasonable reference, so we sketch a proof.

\begin{proof}
	Let $\{U_i\rightarrow U\}$ be a cover in $\cC$. There is a canonical isomorphism
   \[
   \xymatrix@C=45pt{
   A(U) \ar[d]	 \ar[r]^{L(U)} & \cM(U) \ar[d]\\
   A(\{U_i\rightarrow U\})       \ar[r]_{L(\{U_i\rightarrow U\})} \ar@{=>}
   [ur] & \cM (\{U_i\rightarrow U\})  }
   \]
where for a fibered category $\mathcal F\rightarrow \mathcal C$, $\mathcal F(\{U_i\rightarrow U\})$ denotes the category of descent data of $\mathcal F$ with respect to $\{U_i\rightarrow U\}$ (see \cite{fga-explained}). Since $\cM$ is a stack, the right-hand map is an equivalence. Since the diagram above is compatible with refinements of covers, we can take the inductive limit and get a factorization:
\[
\xymatrix{ 
    A(U) \ar[r]|*{}="A"^{L(U)} \ar[d]&  \cM(U)\\ 
    A\sep(U)   \ar[ru]_{L\sep(U)} \ar@{=>};"A" &}
\]
where $A\sep(U)\eqdef \varinjlim_{\{U_i\rightarrow U\}}   A(\{U_i\rightarrow U\})$. This is compatible with restriction, so that we obtain a factorization $L\sep:  A\sep\rightarrow \cM$ of $L:A\rightarrow \cM$, and iterating this process we get the wished factorization $L\sh:  A\sh\rightarrow \cM$ of $L:A\rightarrow \cM$.

This factorization is functorial in $A$, and the existence in the universal property follows, moreover the uniqueness is obvious.

For the last assertion, it is enough to notice that $L\sep$ has trivial kernel if $L$ has.
\end{proof}

If $X$ is a scheme, we denote by $X\et$ the small étale site of $X$, whose objects are étale morphisms $U \arr X$. When we mention a sheaf on $X$, we will always mean a sheaf on $X\et$. Thus, for example, by $\cO_{X}$ we mean the sheaf on $X\et$ sending each $U \to X$ into $\cO(U)$. The Zariski site of $X$ will be hardly used. We will often indicate an object $U \arr X$ of $X\et$ simply by $U$. The restriction of the stacks $\div_{X}$ and $\pic_{X}$ to $X\et$ defined above will be denoted by $\div_{X\et}$ and $\pic_{X\et}$.

\section{\df structures}\label{sec:df-structures}

\subsection{\df structures and logarithmic structures}

Now we reformulate the classical notion of a logarithmic structure on a scheme in a form that is more suitable to define parabolic sheaves.

\begin{definition}\label{def:df-structure}
Let $X$ be a scheme. A \emph{\pdf structure} $(A, L)$ on $X$ consists of the following data:

\begin{enumeratea}

\item a presheaf of monoids $A\colon X\et\op \arr \catmon$ on $X\et$, and

\item a \sm  functor $L\colon A \arr \div_{X\et}$.

\end{enumeratea}

A \emph{\df structure} on $X$ is a \pdf structure $(A, L)$ such that $A$ is a sheaf, and $L$ has trivial kernel.

A morphism of \pdf structures from $(A,L)$ to $(B,M)$ is a pair $(\phi, \Phi)$, where $\phi\colon A \arr B$ is a morphism of sheaves of monoids and $\Phi\colon L \arr M \circ \phi$ is a morphism of \sm cartesian functors. A morphism of \df structures is a morphism of \pdf structures between \df structures.
\end{definition}

Given a sheaf of monoids $A$ on $X\et$, we will sometimes say that a cartesian \sm functor $L\colon A \arr \div_{X}$ \emph{is a \df structure} to mean that $(A, L)$ is a \df structure, that is, that $L$ has a trivial kernel.

The composition of morphisms of \pdf structures on $X$ is defined in the obvious way. This defines the categories of \pdf structures and of \df structures on $X$.

Notice that, since $\div _{X}$ is fibered in groupoids, a morphism $(\phi, \Phi)\colon (A, L) \arr (B, M)$ of \pdf structures is an isomorphism if and only if $\phi\colon A \arr B$ is an isomorphism.

\begin{remark}\label{rmk:conventions-symmetric-monoidal}
To compute with \sm functors $L\colon A \arr \div_X$ the following convention is useful. If $a \in A(U)$, we denote the image of $a$ in $\div U$ by $L(a) = (L_{a}, \sigma^{L}_{a})$. Then $\sigma^{L}_{0} \in L_{0}$ is nowhere vanishing; the isomorphism $\epsilon^L:  (\cO_{U},1)\simeq L(0)$ is uniquely determined by the condition that it carries $\sigma^{L}_{0}$ into $1$.
\end{remark}

The embedding of the category of \df structures on $X$ into the category of \pdf structures has a left adjoint.

\begin{proposition}
	\label{prop:pdf-to-df}
Let $(A,L)$ be a \pdf structure on a scheme $X$. There exists a \df structure $(A\ass, L\ass)$, together with a homomorphism of \pdf structures $(A, L) \arr (A\ass, L\ass)$, that is universal among homomorphism from $(A, L)$ to \df structures.

Furthermore, if $K$ is the kernel of $L\colon A \arr \div_{X}$, then $A\ass$ is the sheafification of the presheaf quotient $A/K$.
\end{proposition}

\begin{proof}
	Thanks to Proposition \ref{prop:factor-predf}, we construct a morphism $(A,L) \arr (A\pass, L\pass)$, such that $A\pass$ is the presheaf quotient $A/K$, and $L\pass$ has trivial kernel, and show that it is universal among morphisms from $(A,L)$ to \pdf structures $(B,M)$ such that the kernel of $M$ is trivial. Then we apply Proposition \ref{prop:sheafification} to sheafify. 
\end{proof}

Next we connect our notion of logarithmic structure with the classical one.

\begin{definition}[\cite{Kato}]
	A \emph{log structure} on a scheme $X$ is a pair $(M,\rho)$, where $M$ is a sheaf of monoids on $X\et$ and $\rho: M\rightarrow \mathcal O_X$ is a morphism, where $\mathcal O_X$ denotes the multiplicative monoid of the ring $\mathcal O_X$, with the property that the induced homomorphism $\rho^{-1} \mathcal O_X^*\rightarrow \mathcal O_X^*$ is an isomorphism.
\end{definition}

A morphism $(M, \rho) \arr (M', \rho')$ is a homomorphism of sheaves of monoids $f\colon M \arr M'$, such that $\rho' \circ f = \rho$. This defines the category of log structures.

This is much too large, and one imposes various conditions on the structure, the first of them usually being that $M$ is an integral sheaf of monoids. An even weaker condition is the following.

\begin{definition}[\cite{ogus-logbook}]
A log structure $(M, \rho)$ on a scheme $X$ is \emph{quasi-integral} if the action of $\rho^{-1}\cO^{*}_{X} \simeq \cO^{*}_{X}$ on $M$ is free.
\end{definition}

In other words, $(M,\rho)$ is quasi-integral if whenever $U \arr X$ is an étale morphism and $\alpha \in \rho^{-1}\cO^{*}_{X}(U)$ and $m \in M(U)$ are such that $\alpha + m = m$, then $\alpha = 0$.

\begin{theorem}\label{thm:df<->log}
The category of \df structures on $X$ is equivalent to the category of quasi-integral log structures on $X$.
\end{theorem}

\begin{proof}
Consider the morphism of \sm fibered categories $\cO_{X} \arr \div_{X\et}$ sending a section $f \in \cO_{X}(U)$ into $(\cO_{U}, f)$. This is a torsor under $\cO^{*}_{X}$. There is a section $X\et \arr \div_{X\et}$ sending $U \arr X$ into $(\cO_{U}, 1)$; this gives an equivalence of $X\et$ with the full fibered subcategory of $\div_{X\et}$ of invertible sheaves with a nowhere vanishing section, and so is represented by Zariski open embeddings. The inverse image of $X\et$ in $\cO_{X}$ is $\cO^{*}_{X}$.

If $(A, L)$ is a \df structure on $X$, define $M$ as the fibered product $A \times_{\div_{X\et}} \cO_{X}$. This is a $\cO_{X}^{*}$-torsor over $A$, hence it is equivalent to a sheaf.  The \sm structures of $A$, $\cO_{X}$ and $\div_{X\et}$ induce a \sm structure on $M$, that is, $M$ acquires a structure of a sheaf of monoids. By hypothesis, the inverse image of $X\et \subseteq \div_{X\et}$ via $L\colon A \arr \div_{X\et}$ is the zero-section $X\et \subseteq A$; hence by base change to $\cO_{X}$ the inverse image of $\cO^{*}_{X} \subseteq \cO_{X}$ in $M$ coincides with the inverse image of $X\et \subseteq A$, which is again $\cO^{*}_{X}$. This shows that $M \arr \cO_{X}$ is in fact a quasi-integral log structure.

This construction gives a functor from \df structures on $X$ to quasi-integral log structures (the action of the functor on arrows is easy to construct). In the other direction, let $\rho\colon M \arr \cO_{X}$ be a quasi-integral log structure; the free action of $\cO^{*}_{X} \simeq \rho^{-1}\cO^{*}_{X}$ makes $M$ into a $\cO^{*}_{X}$-torsor over $\overline{M} \eqdef M/\cO^{*}_{X}$. The homomorphism $\rho\colon M \arr \cO_{X}$ is $\cO^{*}_{X}$-equivariant; the stack-theoretic quotient $[\cO_{X}/\cO^{*}_{X}]$ is $\div_{X\et}$, so we obtain a \sm functor $L\colon \overline{M} \arr \div_{X\et}$. It is immediate to see that the kernel of $L\colon  \overline{M} \arr \div_{X\et}$ is trivial; so $(\overline{M}, L)$ is a \df structure on $X$. This is the action on objects of a functor from quasi-integral log structures to \df structures, which is easily seen a quasi-inverse to the previous functor.
\end{proof}

\subsection{Direct and inverse images of \df structures}\label{subsec:direct-inverse}

Let $f: X'\rightarrow X$ be a morphism of schemes. 
The stack $f_* \div_{X'\et}$ on $X\et$ is defined as the fibered product 
$\div_{X'\et}\times_{X'\et} X\et$, where the functor $X\et \arr X'\et$ is given by fibered product. In other words, if $U \arr X$ is an étale morphism, we have
   \[
   f_* \div_{X'\et}(U) = \div(X' \times_{X} U).
   \]
There is a natural morphism of stacks: $\div_{X\et}\rightarrow f_* \div_{X'\et}$ on $X\et$, given by pullback\footnote{And accordingly there is a natural morphism  of stacks $f^*\div_{X\et}\rightarrow \div_{X'\et}$ on $X'\et$, but since the definition of $f^*\div_{X\et}$ is quite complicated, we won't use it.}. 

Let us begin with the definition of the direct image of a \df structure $(A',L')$ on $X'$.

\begin{lemma}
	
The \sm stack $f_*A'\times_{f_*\div_{X'\et}} \mathcal O_X$, fibered product of $f_*L'\colon f_*A'\rightarrow f_*\div_{X'\et}$ and of the composite $\mathcal O_X\rightarrow \div_{X\et}\rightarrow f_* \div_{X'\et}$, is equivalent to a sheaf of monoids on $X$. 

\end{lemma}
\begin{proof}
	Set $M'=A'\times_{\div_{X'\et}}\mathcal O_{X'}$. The proof of Theorem \ref{thm:df<->log} shows that $M'$ is (equivalent to) a sheaf of monoids on $X'$, and clearly $f_*A'\times_{f_*\div_{X'\et}} \mathcal O_X\simeq f_*M'\times_{f_*\mathcal O_{X'}}\mathcal O_X$. \end{proof}

\begin{definition}
	Let $f: X'\rightarrow X$ be a scheme morphism, and $(A',L')$ a \df structure on $X'$. We define the direct image  $f_*(A',L')$ as the \df structure associated with the \pdf structure $f_*A'\times_{f_*\div_{X'\et}} \mathcal O_X\rightarrow \mathcal O_X \rightarrow \div_{X\et}$\footnote{The morphism $f_*A'\times_{f_*\div_{X'\et}} \mathcal O_X\rightarrow \mathcal O_X$ defines a log structure, but it is not necessarily quasi-integral.}.
\end{definition}

We now define the pull-back of a log structure $(A,L)$ on $X$ along the scheme morphism $f: X'\rightarrow X$.

\begin{proposition}
	\label{prop:pull-back-df-structure}
	There is up to unique isomorphism a unique pair $\left(f^*L,\alpha \right)$ where $f^*L\colon f^*A\rightarrow \div_{X'\et}$ is a \df structure on $X'$ and $\alpha$ an isomorphism of \sm functors between the composites $A\rightarrow f_*f^*A\xarr{f_*f^* L}f_*\div_{X'\et}$ and $A\xarr{L}\div_{X\et}\rightarrow f_*\div_{X'\et}$.
\end{proposition}

\begin{proof}
	The uniqueness statement is easily proven. To show the existence, we can work with the pull-back presheaf $f^{-1}A$, and then use Proposition \ref{prop:sheafification} to sheafify. Let $U'\rightarrow X'$ be an étale morphism, then $f^{-1}A(U')=\varinjlim_{U'\rightarrow U} A(U)$, where $U'\rightarrow U$ varies in all $X$-morphisms from $U'$ to an étale scheme $U\rightarrow X$. Since for each such morphism, we have a monoidal functor $A(U)\xarr{L(U)}\div(U)\arr\div(U')$, and these functors are compatible with restriction, we get a monoidal functor $f^{-1}A(U')\rightarrow  \div(U')$, also compatible with restrictions. This defines a monoidal functor $f^*L:f^*A\rightarrow \div_{X'\et}$, and the existence of $\alpha$ is obvious by construction. The only thing that is left to check is that $f^*L$ has trivial kernel.

	Let $U'\rightarrow X'$ be an étale morphism, and $a'\in f^*A(U')$, such that $f^*L(a')$ is invertible. There exists an étale covering $\{U'_i\arr U'\}_i$ of $U'$, $X$-morphisms $f_i:U_i'\arr U_i$ to an étale scheme $U_i\rightarrow X$, and sections $a_i\in A(U_i)$, such that $a'_{|U'_i}={f_i^*a_i}_{|U'_i}$. Then $f^*L(a')_{|U'_i}$ is invertible, and $f^*L(a')_{|U'_i}\simeq f_i^*L(a_i)$ by the universal property. Let $x:\spec \Omega \rightarrow U'_i$ be a geometric point, we deduce that $x^*(f_i^*(\sigma_{a_i}))$ does not vanish, and so $\sigma_{a_i}$ is invertible on a neighborhood of $f_i(x)$, and since $L$ has trivial kernel, $a_i=0$ on this neighborhood, hence $f_i^*a_i=0$ on a neighborhood of $x$. Thus since $x$ is arbitrary, $f_i^*a_i=0$ on $U'_i$, and this implies $a'_{|U'_i}= 0$, and finally $a'=0$.
\end{proof}

\begin{definition}
	Let $f: X'\rightarrow X$ be a scheme morphism, and $(A,L)$ a \df structure on $X$. The \df structure $f^*L:f^*A\rightarrow \div_{X'\et}$ on $X'$ defined by Proposition \ref{prop:pull-back-df-structure} is called the pull-back \df structure, and denoted by $f^*(A,L)$.
\end{definition}

\begin{remark}
	Proposition \ref{prop:pull-back-df-structure} shows in fact a bit more: there is a canonical adjunction between the functors $(A,L)\mapsto f^*(A,L)$ and $(A',L')\mapsto f_*(A',L')$.
\end{remark}

\subsection{Charts for \df structures}\label{subsec:charts}

If $P$ is a monoid and $X$ is a scheme, we denote by $P_{X}$ the constant presheaf on $X\et$ such that $P_{X}(U) = P$ for all $U \arr X$ in $X\et$, and by $P_{X}\sh$ the associated constant sheaf. Notice that if $A$ is a sheaf of monoids on $X\et$, we have bijective correspondences between homomorphism of monoids $P \arr A(X)$, homomorphism of presheaves of monoids $P_{X} \arr A$, and homomorphism of sheaves of monoids $P_{X}\sh \arr A$.

\begin{definition}
Let $X$ be a scheme, $A$ be a sheaf of monoids on $X\et$. A \emph{chart} for $A$ consists of  a homomorphism of monoids $P \arr A(X)$, such that $P$ is a finitely generated monoid, and the induced homomorphism $P_{X}\sh \arr A$ is a cokernel in the category of sheaves of monoids.

An \emph{atlas} consist of an étale covering $\{X_{i} \arr X\}$ and a chart $P_{i} \arr A(X_{i})$ for each restriction $A]_{X_{i}}$.
\end{definition}

\begin{definition}
A sheaf of monoids $A$ on $X\et$ is \emph{coherent} if it is sharp and has an atlas. A sheaf of monoids is \emph{fine} if it is coherent and integral.

\end{definition}

Being a chart is property that can be checked at the level of stalks at geometric points of $X$.

\begin{proposition}\label{prop:chart-stalks}
Let $X$ be a scheme, $A$ be a sheaf of monoids on $X\et$. A homomorphism of monoids $P \arr A(X)$ is a chart if and only if $P$ is finitely generated, and for each geometric point $x\colon \spec\Omega \arr X$ the induced homomorphism of monoids $P \arr A_{x}$ is a cokernel.
\end{proposition}

\begin{proof}
Let $K$ be the kernel of the induced homomorphism of presheaves of monoids $P_{X} \arr A$. The homomorphism $P_{X}\sh \arr A$ is a cokernel if and only if the induced homomorphism $(P_X/K)\sh \arr A$ is an isomorphism; hence $P_{X}\sh \arr A$ is a cokernel if and only if $P/K_{x} = (P_{X}/K)\sh_{x} \arr A_{x}$ is an isomorphism for all $x$. On the other hand the kernel of $P \arr A_{x}$ is the stalk $K_{x}$, and the stalk of $(P_X/K)\sh$ at $x$ is $P/K_{x}$; so $P/K_{x} \arr A_{x}$ is an isomorphism if and only if $P \arr A_{x}$ is a cokernel.
\end{proof}

In what follows we are going to use Rédei's theorem, stating that every finitely generated commutative monoid is finitely presented (see for example \cite[Theorem~72]{redei}, or \cite{grillet}).

If a sheaf of monoids is coherent, around each geometric point of $X$ there exists a minimal chart.

\begin{proposition}\label{prop:minimal-chart}
Let $A$ be a coherent sheaf of monoids on $X\et$, and let $x\colon \arr \spec\Omega \arr X$ a geometric point. Then there exists an étale neighborhood $\spec\Omega \arr U \arr X$ of $x$ and a chart $P \arr A(U)$ for the restriction $A ]_{U}$, such that the induced homomorphism $P \arr A_{x}$ is an isomorphism.
\end{proposition}

So, for example, a fine sheaf of monoids has an atlas $\bigl(\{X_{i} \to X\}, \{P_{i} \to A(X_{i})\}\bigr)$ in which all the $P_{i}$ are integral and sharp.

\begin{proof}
Let us start with a Lemma.

\begin{lemma}
Let $P$ be a finitely generated monoid. Any cokernel $P \arr Q$ is the cokernel of a homomorphism $F \arr P$, where $F \simeq \NN^{r}$ is a finitely generated free monoid.
\end{lemma}

\begin{proof}
Write $P\rightarrow Q$ as the cokernel of a homomorphism $F \arr P$, where $F$ is a free monoid over a set $I$. For each finite subset $A \subseteq I$ call $F_{A} \subseteq F$ the free submonoid generated by $A$ and $C_{A}$ the cokernel of the composite $F_{A} \subseteq F \arr P$, and $R_{A} \subseteq P \times P$ the congruence equivalence relation determined by the projection $P \arr C_{A}$. The monoidal equivalence relation $R$ determined by the homomorphism $P \arr Q$ is the union of the $R_{A}$; since $R$ is finitely generated as a monoidal equivalence relation, by Rédei's theorem, there exists a finite subset $A \subseteq I$ such that $R = R_{A}$. So $Q$ is the cokernel of $F_{A} \arr P$.
\end{proof}

By passing to an étale neighborhood, we may assume that there exists a global chart $P \arr A(X)$. Consider the homomorphism $P \arr A_{x}$; this is a cokernel, hence by the Lemma there exists a finite free monoid $F$ and a homomorphism $F \arr P$ with cokernel $A_{x}$. By passing to an étale neighborhood, we may assume that the composite $F \arr P \arr A(X)$ is $0$. Call $Q$ the cokernel of $F \arr P$; we have a factorization $P \arr Q \arr A(X)$. If $K$ is the kernel of $P_{X}\sh \arr A$, then we see immediately that $Q_{X}\sh \arr A$ is the cokernel of the composite $K \subseteq P_{X}\sh \arr Q_{X}\sh$. It follows from the construction that the induced homomorphism $Q \arr A_{x}$ is an isomorphism.
\end{proof}

For later use, we note the following fact, saying that charts can be chosen compatibly with arbitrary homomorphisms of coherent sheaves of monoids.

\begin{proposition}\label{prop:compatible-charts}
Let $f\colon A \arr B$ be a homomorphism of coherent sheaves of monoids on a scheme $X$. Given a geometric point $x\colon \spec\Omega \arr X$, there exists an étale neighborhood $U \arr X$ of $x$, two finitely generated monoids $P$ and $Q$, and a commutative diagram
   \[
   \xymatrix{
   P \ar[d]\ar[r] & Q \ar[d]\\
   A(U) \ar[r]^{f(U)}    & B(U)
   }
   \]
where the columns are charts for $f|_{U}\colon A]_{U} \arr B]_{U}$. Furthermore, the columns can be chosen so that the induced homomorphisms $P \arr A_{x}$ and $Q \arr B_{x}$ are isomorphisms.
\end{proposition}

\begin{proof}
After passing to an étale neighborhood, we may assume that there are charts $P \arr A(X)$ and $Q \arr B(X)$ such that the composites $P \arr A(X) \arr A_{x}$ and $Q \arr B(X) \arr B_{x}$ are isomorphisms, by Proposition~\ref{prop:minimal-chart}. The homomorphism $f_{x}\colon A_{x} \subseteq B_{x}$ induces a homomorphism $P \arr Q$. The diagram above does not necessarily commute: however the images of a given finite number of generators of $P$ in $B_{x}$ coincide, hence after a further restriction we may assume that it does commute.
\end{proof}

\begin{proposition}
Let $f\colon Y \arr X$ be a morphism of schemes, $A$ a coherent sheaf of monoids on $X\et$. Then the sheaf $f^{*}A$ on $Y\et$ is coherent.
\end{proposition}

\begin{proof}
Being coherent is a local property in the étale topology, so we may assume that there exists a chart $P \arr A(X)$. We claim that the composite $P \arr A(X) \xarr{f^{*}} f^{*}A(Y)$ is also a chart. According to Proposition~\ref{prop:chart-stalks} this can be checked at the level of stalks: but the stalk of $f^{*}A$ at a geometric point $y$ of $Y$ is the stalk of $A$ at the image of $y$, so the statement is clear.
\end{proof}

The condition of being coherent is local in the fppf topology. In fact we have the following stronger statement.

\begin{proposition}\label{prop:coherent-local}
Let $f\colon Y \arr X$ be an open surjective morphism of schemes, and let $A$ be a sheaf of monoids on $X\et$. If $f^{*}A$ is coherent as a sheaf on $Y\et$, then $A$ is coherent.
\end{proposition}

\begin{proof}
Let $x\colon \spec \Omega \arr X$ be a geometric point; after possibly extending $\Omega$, we may assume that there exists a lifting $y\colon \spec\Omega \arr Y$. After passing to an étale neighborhood of $y$ in $Y$ and replacing $X$ with its image in $X$, by Proposition~\ref{prop:minimal-chart} we may assume that there exists a chart $P \arr f^{*}A(Y)$. From the induced homomorphism $P \arr f^{*}A_{y}$ and the canonical isomorphism of stalks $A_{x} \simeq f^{*}A_{y}$ we obtain a homomorphism $P \arr A_{x}$. Since by Rédei's theorem $P$ is finitely presented, after passing to an étale neighborhood of $x$ we can assume that the homomorphism $P \arr A_{x}$ factors as $P \arr A(X) \arr A_{x}$. The composite of $P \arr A(X)$ with the canonical homomorphism $A(X) \arr f^{*}A(Y)$ does not necessarily coincide with the given chart $P \arr f^{*}A(Y)$; but since the images of the generators of $P$ in $(f^*A)_{y}$ through the two maps are the same, after passing to an étale neighborhood of $y$ in $Y$ and further shrinking $X$ we may assume that the two homomorphisms $P \arr A(Y)$ coincide.

We claim that $P \arr A(X)$ is a chart for $A$. If $K$ is the kernel of the homomorphism of presheaves of monoids $P_{X} \arr A$, we need to check that the induced homomorphism $(P_{X}/K)\sh \arr A$ is an isomorphism, or, equivalently, that for any geometric point $\xi$ of $X$ the induced homomorphism $P/K_{\xi} = (P_{X}/K)\sh_{\xi} \arr A_{\xi}$ is an isomorphism. But if $\eta$ is a geometric point of $Y$ lying over $\xi$, we have that the kernel of $P_{Y} \arr f^{*}A$ is the presheaf pullback $f^{\rmp}K$; so $K_{\xi} = f^{\rmp}K_{\eta} \subseteq P$. The induced homomorphism $P/f^{\rmp}K_{\eta} \arr f^{*}A_{\eta} \simeq A_{\xi}$ is an isomorphism, and this completes the proof.
\end{proof}

\begin{definition}
A \df structure $(A, L)$ on a scheme $X$ is \emph{coherent} if $A$ is a coherent sheaf of monoids. It is \emph{fine} if $A$ is fine (i.e., integral and coherent).
\end{definition}

Let $(A, L)$ be a \df structure on a scheme $X$, $P \arr A(X)$ a chart. The composite $L_{0}\colon P \arr A(X) \xarr{L(X)} \div X$ completely determines the \df structure.

\begin{proposition}\label{prop:chart->df}
Let $X$ be a scheme, $P$ a finitely generated monoid, $L_{0}\colon P \arr \div X$ a \sm functor. Then there exists a \df structure $(A, L)$ on $X$, together with a homomorphism of monoids $\pi\colon P \arr A(X)$ and an isomorphism $\eta$ of \sm functors between $L_{0}$ and the composite
   \[
   P \stackrel{\pi}{\larr} A(X)
      \xarr{L(X)} \div X
   \]
such that $P \arr A(X)$ is a chart. If $K$ is the kernel of the \sm functor $P_{X} \arr \div_{X\et}$, then $A$ is isomorphic to $(P_{X}/K)\sh$.

Furthermore, this is universal among such homomorphisms.

More precisely, given $(A', L')$ another \df structure on $X$, a homomorphism of monoids $\pi'\colon P \arr A'(X)$, and an isomorphism $\eta'$ of $L_{0}$ with the composite $P \xarr{\pi'} A'(X) \xarr{L'(X)} \div X$, there exists a unique morphism $(\phi, \Phi)\colon (A, L) \arr (A', L')$ such that $\phi(X)\circ\pi =\pi'$ and the diagram
   \[
   \xymatrix{
   L_{0} \ar[r]^-{\eta} \ar[rd]^-{\eta'} 
   & L(X) \circ \pi \ar[d]\\
   &L'(X) \circ \pi'
   }
   \]
of \sm functors $P \arr \div X$, where the vertical arrow is the homomorphism induced by $\Phi_{X}\colon L(X) \arr L'(X)\circ \phi(X)$, commutes.

We call $(A, L)$ \emph{the \df structure associated with $L_0$}.
\end{proposition}

\begin{proof}
This is a direct consequence of Proposition \ref{prop:pdf-to-df}.
\end{proof}

\begin{example}\label{def:generated-df}
Let $X$ be a scheme, $(L_{1}, s_{1})$, \dots,~$(L_{r}, s_{r})$ objects of $\div X$. Consider the monoidal functor $L_{0}\colon \NN^{r} \arr \div X$ that sends $(k_{1}, \dots, k_{r}) \in \NN^{r}$ into
   \[
   (L_{1}, s_{1})^{\otimes k_{1}}
   \otimes \dots\otimes (L_{r}, s_{r})^{\otimes k_{r}}\,.
   \]
We call the \df structure associated with $L_{0}$ \emph{the \df structure generated by $(L_{1}, s_{1})$, \dots, $(L_{r}, s_{r})$}.

Denote by $e_{i} \in \NN^{r}$ the $i\th$ canonical basis vector, that is, the vector with $1$ at the $i\th$ place and $0$ everywhere else. Suppose that $L'_{0}\colon \NN^{r} \arr \div X$ is another \sm functor, and for each $i = 1$, \dots,~$r$ we have an isomorphism $\phi_i\colon L'_{0}(e_{i}) \simeq (L_{i}, s_{i})$; then it is easy to show that there exists a unique isomorphism $\phi\colon  L'_{0} \simeq L_{0}$ whose value at $e_{i}$ is $\phi_{i}$. By Proposition~\ref{prop:chart->df}, this implies that the \df structure $(A,L)$ generated by $(L_{1}, s_{1})$, \dots,~$(L_{r}, s_{r})$ is, up to a unique isomorphism, the only \df structure with a chart $\NN^{r} \arr A(X)$, such that the composite $\NN^{r} \arr A(X) \xarr{L(X)} \div (X)$ sends $e_{i}$ to $(L_{i}, s_{i})$.
\end{example}

\subsection{Charts and Kato charts}\label{subsec:kato-charts}

Our notion of chart can be compared with Kato's.

\begin{definition}[\cite{Kato}]
A \emph{Kato chart} for the log structure $(M,\rho)$ is the data of a finitely generated monoid $P$, and a morphism $P\rightarrow M(X)$, such that the composite $P \arr M(X) \arr \overline{M}(X)$ is a chart for $\overline{M}$.
\end{definition}

\begin{definition}[{\cite[Definition~2.1.1]{ogus-logbook}}]
A log structure admitting a chart locally on $X\et$ is called \emph{coherent}. A log structure is \emph{fine} if it is coherent and integral.
\end{definition}

If we are given a finitely generated monoid $P$ and a homomorphism $P \arr \cO(X)$, we compose with the morphism $\cO(X) \arr \div X$ defined in the proof of Theorem~\ref{thm:df<->log}, sending $f \in \cO(U)$ into $(\cO_{U}, f)$, we obtain a \sm functor $P \arr \div X$, which, according to Proposition~\ref{prop:chart->df}, gives us a \df structure $(A, L)$ on $X$. Call $(M, \rho)$ the associated log structure: the homomorphism $P \arr A(X)$, together with $P \arr \cO(X)$, yields a Kato chart $P \arr M(X)$ (recall, from the proof of Theorem~\ref{thm:df<->log}, that $M \eqdef A\times_{\div_{X\et}}\cO_{X}$). So the log structure associated with a \df structure $(A, L)$ is coherent if and only if there exists an étale cover $\{X_{i} \arr X\}$ and charts $\{P_{i} \arr A(X_{i})\}$, such that the composites $P_{i} \arr A(X_{i}) \xarr{L} \div X_{i}$ lift to homomorphisms $P_{i} \arr \cO(X_{i})$.

Now we want to investigate the question of when a chart $P \arr A(X)$ for a \df structure $(A, L)$ on $X$ lifts to a Kato chart $P \arr M(X)$. In other words, when does a \sm functor $P \arr \div X$ lift to a homomorphism of monoids $P \arr \cO(X)$?

Fix a finitely generated monoid $P$. We denote by $\ZZ[P]$ the monoid ring of $P$. Since we are using additive notation for $P$, it is convenient to introduce an indeterminate $x$, and write an element of $\ZZ[P]$ as a finite sum $\sum_{p \in P}a_{p}x^{p}$, where $a_{p} \in \ZZ$ for all $p$.

A morphism $P \arr \cO(X)$ correspond to a ring homomorphism $\ZZ[P] \arr \cO(X)$, hence to a morphism of schemes $X \arr \sz{P}$. Thus we think of $\sz{P}$ as representing the functor $\underhom(P, \AA^{1})$ from schemes to monoids (the monoidal structure is given by multiplication in $\AA^{1}$). Thus $\sz{P}$ is the space of Kato charts.

Consider the fibered category $\underhom(P, \div_{\ZZ}) \arr \cat{Sch}$, whose objects over a scheme $X$ are \sm functors $P \arr \div X$. A morphism $X \arr \underhom(P, \div_{\ZZ})$ gives a chart for a \df structure on $X$. We think of $\underhom(P, \div_{\ZZ})$ as the stack of charts. There is an obvious morphism
   \[
   \underhom(P, \AA^{1}) \arr \underhom(P, \div_{\ZZ})
   \]
that corresponds to the procedure of associating a chart to a Kato chart. In other words, if a morphism $X \arr \underhom(P, \AA^{1})$ corresponds to a homomorphism $P \arr \cO(X)$, the corresponding morphism $X \arr \underhom(P, \div_{\ZZ})$ corresponds to the composite $P \arr \cO(X) \arr \div X$. The issue is: when is it possible to lift a \sm functor $X \arr \underhom(P,\div_{\ZZ})$ to a \sm functor $X \arr \sz{P}$?

Set
   \[
   \widehat{P} \eqdef \underhom(P, \gm) = \underhom(P\gp, \gm)\,;
   \]
then $\widehat{P}$ is a diagonalizable group scheme, acting on $\sz{P}$ (the action is induced by the action of $\gm$ on $\AA^{1}$ by multiplication). Equivalently, we can think of $\widehat{P}$ as the group scheme of invertible elements of $\sz{P}$.

Since the group scheme $\widehat{P}$ is diagonalizable, with character group
   \[
   P\gp \simeq \hom(\widehat{P}, \gm)\,,
   \]
by the standard description of representations of $\widehat{P}$, a $\widehat{P}$-torsor $\eta\colon E \arr T$ gives a $P\gp$-grading on the sheaf of algebras $\eta_{*}\cO_{E}$. The trivial torsor $\widehat{P}\times T \arr T$ corresponds to the group algebra $\cO_{T}[P\gp]$. This gives an equivalence of categories between the category of $\widehat{P}$-torsors and the opposite of the groupoid of sheaves of $P\gp$-graded algebras over $\cO_{T}$, such that each graded summand is invertible. Given such an algebra $A$, the torsor $E$ is the relative spectrum $\underspec_{T}A$; the action is defined by the grading.

 The action of $\widehat{P}$ on $\spec\ZZ[P]$ corresponds to the natural $P\gp$-grading
   \[
   \ZZ[P]
   = \bigoplus_{u \in P\gp}\Bigl(\bigoplus_{\substack{p \in P\\\iota_{P}(p) = u}}
   \ZZ x^{p}\Bigr)\,.
   \]
A morphism $T \arr \szp{P}$ corresponds to a $\widehat{P}$-torsor $\eta\colon E \arr T$ and a $\widehat{P}$-equivariant morphism $E \arr \sz{P}$; this morphism gives $\eta_{*}\cO_{E}$ the structure of a $\cO_{T}[P]$ algebra, which is compatible with the $P\gp$-grading of $\ZZ[P]$. Hence, we see that the groupoid of objects of $\szp{P}$ over $T$ is equivalent to the opposite of the groupoid whose objects are sheaf of commutative $P\gp$-graded $\cO_{T}[P]$-algebras over $T$, whose grading is compatible with the grading of  $\cO_{T}[P]$, that are fppf locally isomorphic to $\cO_{T}[P\gp]$ as graded $\cO_{T}$-algebras.

\begin{proposition}\label{prop:stack-of-charts}
Let $P$ be a finitely generated monoid. There is an equivalence of fibered categories between $\underhom(P, \div_{\ZZ})$ with the quotient stack $\szp{P}$ by the action defined above.
\end{proposition}

\begin{proof}
Let $T$ be a scheme; suppose that we are given an object of $\szp{P}$ over $T$, corresponding to a sheaf $A$ of $\cO_{T}[P]$-algebras, as above. With this, we associate a \sm functor $P \arr \div T$ as follows. Write $A = \oplus_{u \in P\gp}A_{u}$; then by the local description of $A$ we see that $A_{u}$ is an invertible sheaf on $T$. The functor $P \arr \div T$ associates with $p \in P$ the pair $(A_{\iota_{P}(p)}, x^{p})$, where by abuse of notation we identify the element $x^{p} \in \ZZ[P]$ with its image in $A(T)$. The \sm structure on the functor is given by the algebra structure on $A$; we leave the easy details to the reader.

For the inverse construction, we need the following. Suppose that $G$ is a finitely generated abelian group, $L\colon G \arr \pic T$ a \sm functor. With this we can associate a $G$-graded sheaf of $\cO_{T}$-modules
   \[
   A^{L} \eqdef \bigoplus_{g \in G} L_{g}.
   \]
It is easy to see that the isomorphisms $L_{g} \otimes_{\cO_{T}} L_{h} \arr L_{g+h}$ coming from the \sm structure of $L$ give $A^{L}$ the structure of a sheaf of commutative $G$-graded algebras.

The trivial \sm functor $G \arr \div T$ defines the sheaf of group algebras $\cO_{T}[G]$.

\begin{lemma}\label{lem:locally-trivial}
Locally in the fppf topology, the sheaf of commutative $G$-graded algebras associated with a \sm monoidal functor $G \arr \div T$ is isomorphic to $\cO_{T}[G]$.
\end{lemma}

\begin{proof}
Let us decompose $G$ as a product $G_{1} \times \dots \times G_{r}$ of cyclic groups. If we denote by $L^{i}$ the restriction of $L$ to $G_{i} \subseteq G$, it is immediate to see that if $g_{i} \in G_{i}$ for all $i$, then
$L_{g_{1} \dots g_{r}} \simeq L_{g_{1}} \otimes_{\cO_{T}} \dots \otimes_{\cO_{T}} L_{g_{r}}$, and that this induces an isomorphism of sheaves of $\cO_{T}$-algebras
   \[
   A^{L} \simeq A^{L^{1}} \otimes_{\cO_{T}} \dots \otimes_{\cO_{T}} A^{L^{r}}\,;
   \]
so we may assume that $G$ is cyclic.

Call $\gamma$ a generator of $G$. If $G$ is infinite, after restricting $T$ in the Zariski topology we may assume that there exists an isomorphism $L_{\gamma} \simeq \cO_{T}$. The monoidal structure of $L$ gives an isomorphism $L_{\gamma^{k}} \simeq L_{\gamma}^{\otimes k}$ for any element $\gamma^{k}$ in $G$; this induces an isomorphism of sheaved of $G$-graded modules $A^{L} \simeq \bigoplus_{g \in G}\cO_{T} = \cO_{T}[G]$, which is immediately seen to be an isomorphism of $\cO_{T}$-algebras.

If $G$ has order $n$, then the \sm structure of $L$ gives isomorphism
   \[
   \cO_{T} \simeq L_{1} = L_{\gamma^{n}} \simeq L_{\gamma}^{\otimes n}\,;
   \]
after passing to a fppf cover of $T$, we may assume that there exists a nowhere vanishing section $s$ of $L_{\gamma}$ such that $s^{\otimes n}$ corresponds to $1 \in \cO_{T}$; this defines an isomorphism $\cO_{T} \simeq L_{\gamma}$, sending $1$ to $s$. If $\gamma^{k} \in G$, we obtain an isomorphism
   \[
   \cO_{T}\simeq L_{\gamma}^{k} \simeq L_{\gamma^{k}}\,;
   \]
the condition on $s$ ensures that this is independent of $k$. As in the previous case, this defines an isomorphism $A^{L} \simeq \cO_{T}[G]$, which is easily seen to be an isomorphism of algebras.
\end{proof}

We claim that the functor from $\szp{P}(T) \arr \hom(P, \div T)$ constructed above is an equivalence. Given a \sm  functor $L\colon P \arr \div T$ we first define a \sm functor $L\gp\colon P \arr \pic T$ by the obvious formula\footnote{See Proposition \ref{prop:universal-to-pic-stack} for a more general construction.} $L\gp_{\iota_{P} a-\iota_{P} b}\eqdef L_a\otimes L_b^\vee$, and then construct a sheaf of algebras $A \eqdef \bigoplus_{u \in P\gp} L\gp_{u}$. We define a structure of $\cO_{T}[P]$-algebra on $A$ by sending, for each $p \in \cO[P](T)$ the element $x^{p}$ into the tautological section $\sigma^{L}_{p} \in L_{p} \simeq L\gp_{\iota_{P}p}$. We need to check that $A$ gives an object of $\szp{P}$ over $T$; once this is done, it is straightforward to verify that this construction gives a quasi-inverse to the functor defined above. In doing so the only difficulty is to show that $A$ is fppf locally isomorphic to $\cO_{T}[P\gp]$; and this is the content of the Lemma~\ref{lem:locally-trivial}. This concludes the proof of the Proposition.
\end{proof}

\begin{corollary}\label{cor:chartdf<->chartkato}
Let $(A, L)$ be \df structure on a scheme $X$, and call $(M,\rho)$ the induced log structure. Then $(A,L)$ is coherent if and only if there exists a fppf cover $X' \arr X$ such that the pullback of $(M,\rho)$ to $X'$ is coherent.
\end{corollary}

\begin{proof}
Assume that $(M, \rho)$ becomes coherent after pulling to a fppf cover $f\colon X' \arr X$. Then $f^{*}\overline{M} = f^{*}A$ is coherent, and, by Proposition~\ref{prop:coherent-local}, $(A, L)$ is coherent.

On the other hand, if $(A, L)$ is coherent, pick an étale cover $\{X_{i} \arr X\}$ and charts $\{P_{i} \arr A(X_{i})\}$. Consider the induced morphisms $X_{i} \arr \underhom(P_{i}, \div_{\ZZ})$. The pullback
   \[
   X_{i}' \eqdef X_{i} \times_{\underhom(P_{i}, \div_{\ZZ})} \sz{P_{i}}
   \]
is a $\widehat{P}_{i}$-torsor over $X_{i}$, and the pullback of $(M, \rho)$ to $X_{i}'$ is coherent. We conclude the proof by setting $X' \eqdef\sqcup_{i}X_{i}'$.
\end{proof}

But in fact we can do better.

\begin{proposition}\label{prop:coherent<->coherent}
Let $(A, L)$ be \df structure on a scheme $X$, and call $(M,\rho)$ the induced log structure. Then $(A,L)$ is coherent if and only if $(M,\rho)$ is coherent.
\end{proposition}

From this and from Corollary~\ref{cor:chartdf<->chartkato} we obtain the following, which seems to be new.

\begin{corollary}
Let $(M,\rho)$ be a logarithmic structure on a scheme $X$. If there exists a fppf cover $X' \arr X$ such that the pullback of $(M,\rho)$ to $X'$ is coherent, then $(M, \rho)$ is coherent.
\end{corollary}

\begin{proof}[Proof of Proposition~\ref{prop:coherent<->coherent}]
If $(M, \rho)$ is coherent, then $(A,L)$ is coherent, by Corollary~\ref{cor:chartdf<->chartkato}; so we may assume that $(A,L)$ is coherent. We need to show that $(M,\rho)$ is coherent.

First of all, consider the case that $X$ is the spectrum of a strictly henselian local ring $R$. Then the global sections functor gives an equivalence between the category of sheaves of monoids on $X\et$ and that of monoids; consequently, we can identify $M$ and $A$ with their monoids of global sections.

We will show that there exists a finitely generated submonoid $P \subseteq M$, such that the composite $P\subseteq M \arr A$ is a cokernel; then the embedding $P \subseteq M$ gives a Kato chart. This is done as follows.

The kernel of the natural projection $M \arr A$ is the group $R^{*}$ of units in $R$. The monoid $A$ is finitely generated; let $a_{1}$, \dots,~$a_{s}$ be generators, and let $q_{1}$, \dots,~$q_{s}$ be elements of $M$ mapping to $a_{1}$, \dots,~$a_{s}$ respectively. Denote by $Q$ the submonoid of $M$ generated by the $q_{i}$'s. Denote by $S$ the image of the induced homomorphism $Q\gp \arr M\gp$; since $Q$ is a finitely generated monoid, the group $S$ is finitely generated, and so is the subgroup $S \cap R^{*}$ of $R^{*}$.

Let $r_{1}$, \dots,~$r_{t}$ be generators of the group $S \cap R^{*}$; denote by $P$ the submonoid of $M$ generated by $q_{1}$, \dots,~$q_{s}$, $\pm r_{1}$, \dots,~$\pm r_{t}$. We claim that $P \cap R^{*} = S \cap R^{*}$.

The inclusion $S \cap R^{*}\subseteq P \cap R^{*}$ is obvious. Let $p \in P \cap R^{*}$; write $p = q + p'$, where $q \in Q$ and $p' \in S \cap R^{*}$. Then $q = p - p' \in R^{*}$, hence $q \in Q\cap R^{*} \subseteq S \cap R^{*}$, so $p \in S \cap R^{*}$.

Let us verify that the composite $P\subseteq M \arr A$ is a cokernel. It is clearly surjective, since $Q$, which is contained in $P$, surjects onto $A$. Now we need to check that if $p_{1}$ and $p_{2}$ are elements of $P$ having the same image in $A$, then there exists $m \in \ker(P \arr A) = S \cap R^{*}$ such that $p_{2} = p_{1} + m$. Such an $m$ exists in $R^{*}$, because $M/R^{*} = A$. It is immediate to see that the image of $P\gp$ in $M\gp$ equals $S$; hence $m \in S$, and this concludes the proof.

In the general case, let $x\colon \spec \Omega \arr X$ be a geometric point of $X$; we need to construct a chart for $(M,\rho)$ in some étale neighborhood of $x$ in $X$. Let $R$ be the strict henselization of the local ring $\cO_{X,x}$; by the previous case, the pullback $(M_{x}, \rho_{x})$ of $(M, \rho)$ to $\spec R$ is coherent. It is easy to see that $M_{x}$ is the fiber of $M$ at $x$. Let $P \arr M_{x}$ be a Kato chart; since $P$ is finitely presented, after passing to an étale neighborhood of $x$ we may assume that $P \arr M_{x}$ comes from a homomorphism $P \arr M(X)$. We need to check that the composite $P_{X}\sh \arr M \arr A$ is a cokernel, perhaps after further restricting $X$.

Let $K$ be the kernel of $P_{X}\sh \arr A$; we have to show that the induced homomorphism $f\colon P_{X}\sh/K \arr A$ is an isomorphism. Set $B \eqdef P_{X}\sh/K$.  The kernel of $f$ is obviously trivial; this, together with the fact that $A$ is sharp, implies that $B$ is sharp. Also, $B$ has a tautological chart, hence it is coherent. From Proposition~\ref{prop:compatible-charts}, we see that after restricting to an étale neighborhood of $x$ we may assume that there are charts $B_{x} \arr B(X)$ and $A_{x} \arr A(X)$, inducing the identity on $B_{x}$ and $A_{x}$, such that the diagram
   \[
   \xymatrix{
   B_{x} \ar[r]^{f_{x}} \ar[d] & A_{x}\ar[d]\\
   B(X) \ar[r]^{f(X)} & A(X)
   }
   \]
commutes. Call $K_{B}$ and $K_{A}$ the kernels of the induced homomorphisms $(B_{x})_{X}\sh \arr B$ and $(A_{x})_{X}\sh \arr A$ respectively; since the kernel of $f$ is trivial, we see that $K_{B} = f_{x}^{-1}K_{A}$. Since $f_{x}$ is an isomorphism, it induces an isomorphism $(B_{x})_{X}\sh/K_{B} \simeq (A_{x})_{X}\sh/K_{A}$; but the induced homomorphisms $(B_{x})_{X}\sh/K_{B} \arr B$ and $(A_{x})_{X}\sh/K_{B} \arr A$ are isomorphisms by hypothesis, so $f$ is an isomorphism, as claimed.
\end{proof}

\section{Systems of denominators and stacks of roots}\label{sec:sys-roots}

Several instances of the construction given in this section have been introduced by Martin Olsson (\cite{MO}, \cite{olsson-log-twisted}), and the idea should certainly be attributed to him.

\subsection{Systems of denominators}\label{subsec:sys}
\begin{definition}
Let $A$ and $B$ be finitely generated monoids. A \emph{Kummer homomorphism} $f\colon A \arr B$ is an injective homomorphism of monoids, such that for each $b \in B$ there exists a positive integer $m$ such that $mb$ is in the image of $f$.
\end{definition}

\begin{lemma}\label{lem:group-Kummer}
Let $f\colon A \arr B$ be a Kummer homomorphism of finitely generated monoids. Then $f\gp \colon A\gp \arr B\gp$ is injective with finite cokernel.
\end{lemma}

\begin{proof}
Since $B\gp$ is a finitely generated abelian group, to show that the cokernel is finite it is enough to show that it is torsion. This is evident.

Let us show that $f\gp$ is injective. We write the elements of $A\gp$ as differences $a - a'$; we have $a - a' = 0$ in $A\gp$ if and only if there exists $x \in A$ such that $a + x = a' + x$ in $A$. Suppose that $f(a) - f(a') = f\gp(a-a') = 0$. Then there exists $y \in B$ such that $f(a) + y = f(a') + y$. If $m > 0$ is such that $my = f(x)$, we have $f(a + x) = f(a) + my = f(a') + my = f(a' + x)$ in $B$, hence $a + x = a' + x$ in $B$, and $a - a' = 0$ in $A\gp$.
\end{proof}

\begin{definition}
Let $X$ be a scheme, $A$ a coherent sheaf of monoids on $X\et$. A \emph{system of denominators} for $A$ is an injective homomorphism of sheaves of monoids $A \arr B$ such that

\begin{enumeratea}

\item For any geometric point $x$ of $X$, the induced homomorphism $A_{x} \arr B_{x}$ is a Kummer homomorphism, and

\item $B$ is coherent.

\end{enumeratea}
\end{definition}

Obviously, a system of denominators $A \arr B$ is injective, because it is injective stalkwise. We will normally write a system of denominators for $A$ as $B/A$, and think of $A$ as a subsheaf of $B$.

\begin{example}
Suppose that $A$ is a sharp integral torsion-free sheaf of monoids, $d$ a positive integer. Then the homomorphism $A \arr A$ sending $a$ to $da$ is a system of denominators for $A$.
\end{example}

\begin{definition}
Let $X$ be a scheme, $A$ a coherent sheaf of monoids on $X\et$, $A \subseteq B$ a system of denominators. A \emph{chart} for $B/A$ is a commutative diagram of monoids
   \begin{equation}\label{eq:chart-for-sod}
      \xymatrix{
   P \ar[d]\ar[r] & Q \ar[d]\\
   A(X) \ar[r]    & B(X)
   }
   \end{equation}
where the bottom row is induced by the embedding $A \subseteq B$, the top row is a Kummer homomorphism, and the columns are charts for $A$ and $B$ respectively.
\end{definition}

\begin{proposition}
Let $B/A$ be a system of denominators for the coherent sheaf of monoids $A$ on $X\et$. Given a geometric point $x\colon \spec\Omega \arr X$, there exists an étale neighborhood $U \arr X$ of $x$ such that the restriction $A]_{U} \arr B]_{U}$ has a chart
   \[
   \xymatrix{
   P \ar[d]\ar[r] & Q \ar[d]\\
   A(U) \ar[r]    & B(U)\,.
   }
   \]
Furthermore, the chart can be chosen so that the induced homomorphisms $P \arr A_{x}$ and $Q \arr B_{x}$ are isomorphisms.
\end{proposition}

\begin{proof}
This follows immediately from Proposition~\ref{prop:compatible-charts}.
\end{proof}

\begin{lemma}\label{lem:char-kernel-chart}
Let $A$ be a coherent sheaf of monoids on $X\et$,
   \[
   \xymatrix{
   P \ar[r]^{\phi}\ar[d] &Q \ar[d]\\
   A(X) \ar[r] & B(X)
   }
   \]
a chart for a system of denominators $B/A$. Denote by $K_{A}$ and $K_{B}$ the kernel of the induced homomorphisms $P_{X} \arr A$ and $Q_{X} \arr B$ respectively. Let $U \arr X$ be an étale morphism. An element $q$ of $Q$ is in $K_{B}(U)$ if an only if there exists a covering $\{U_{i} \arr U\}$ and positive integers $m_{i}$ such that $m_{i}q]_{U_{i}} \in K_{A}(U_{i})$.
\end{lemma}

\begin{proof}
The proof is easy and left to the reader.
\end{proof}

\begin{remark}
As a corollary of this fact we have that the system of denominators $B/A$ is uniquely determined up to a unique homomorphism by the chart $P \arr A(X)$ and the Kummer homomorphism $P \arr Q$, since $B \simeq (Q_{X}/K_{B})\sh$, and $K_{B}$ does not depend on $B$.

On the other hand, if we are given a chart $P \arr A(X)$ and a Kummer homomorphism $P \arr Q$, this does not necessarily give a chart for a system of denominators $B/A$. The problem is that if we define $K_{B} \subseteq Q_{X}$ by the formula of Proposition~\ref{lem:char-kernel-chart}, the induced homomorphism $A \simeq (P_{X}/K_{A})\sh \arr (Q_{X}/K_{B})\sh$ is not necessarily injective, in this generality.

It is easy to show that $A \arr (Q_{X}/K_{B})\sh$ injective, for example, when $P$ and $Q$ are integral and saturated, which is the case of greatest interest for the applications. With these hypotheses we can conclude that given a chart $P \arr A(X)$ and Kummer homomorphism $P \arr Q$, there exists a system of denominators $B/A$, together with a chart as above. Furthermore, $B/A$ is uniquely determined up to a unique homomorphism.
\end{remark}

\subsection{Stacks of roots}

Let us start by defining categories of roots for \df objects.

\begin{definition}\label{def:category-roots1}

Let $j\colon P \arr Q$ a homomorphism of monoids, and let $\cM$ a \sm category, $L\colon P \arr \cM$ a \sm functor. Then we define the \emph{category of roots} $(L)(Q/P)$ as follows.

Its objects are pairs $(M, \alpha)$, where $M\colon Q \arr \cM$ is a \sm functor, and $\alpha\colon L \arr M \circ j$ is an isomorphism of \sm functors from $L$ to the composite $M \circ j\colon P \arr \cM$.

An arrow $h$ from $(M, \alpha)$ to $(M', \alpha')$ is an isomorphism $h\colon M \arr M'$ of \sm functors $Q \arr \cM$, such that the diagram
   \[
   \xymatrix{
   &L\ar[rd]^{\alpha'}\ar[ld]_{\alpha}\\
   M \circ j\ar[rr]^{h \circ j} && M'\circ j
   }
   \]
commutes.
\end{definition}

The name \emph{category of roots} is only justified when $j$ is a Kummer morphism; but this hypothesis is not required at this stage.

Here is a more general definition.

\begin{definition}\label{def:category-roots2}
Let $\cC$ be a category, $j\colon A \arr B$ a homomorphism of presheaves of monoids $\cC\op \arr \catmon$. Let $\cM \arr \cC$ a \sm fibered category, $L\colon A \arr \cM$ a \sm morphism of fibered categories. Then we define the \emph{category of roots} $(\cC, L)(B/A)$ as follows.

The objects are pairs $(M, \alpha)$, where $M\colon B \arr \cM$ is a \sm functor, and $\alpha\colon L \arr M \circ j$ is an isomorphism of \sm functors from $L$ to the composite $M \circ j\colon A \arr \cM$.

An arrow $h$ from $(M, \alpha)$ to $(M', \alpha')$ is an isomorphism $h\colon M \arr M'$ of \sm functors $B \arr \cM$, such that the diagram
   \[
   \xymatrix{
   &L\ar[rd]^{\alpha'}\ar[ld]_{\alpha}\\
   M \circ j\ar[rr]^{h \circ j} && M'\circ j
   }
   \]
commutes.
\end{definition}

When $\cC$ is the category with one object and one morphism, we recover the previous definition.

\begin{remark}
Notice that categories of roots, as defined above, are groupoids.
\end{remark}

Next we define stacks of roots for \df structures in two different contexts. 

Suppose that $X$ is a scheme, $j\colon P\arr Q$ a homomorphism of monoids, $L\colon P \arr \div X$ a \sm functor. For each morphism of schemes $t\colon T \arr X$, the pullback $t^{*}\colon \div X \arr \div T$ yields a \sm  functor $t^{*}\circ L\colon P \arr \div T$, from which we obtain a category of roots $(t^{*}\circ L)(Q/P)$.

Let $f\colon T' \arr T$ be a homomorphism of $X$-schemes from $t'\colon T' \arr X$ to $t\colon T \arr X$. Suppose that $(M, \alpha)$ is an object of the category of roots $(t^{*}\circ L)(Q/P)$. Then the composite $f^{*}\circ M\colon P \arr \div T'$ is a \sm functor. The isomorphism $\alpha\colon t^{*}\circ L \simeq M\circ j$ can be pulled back along $f$, yielding an isomorphism
   \[
   f^{*}\circ \alpha\colon f^{*} \circ t^{*}\circ L
   \simeq f^{*} \circ M \circ j \,;
   \]
by composing with the natural isomorphism $t'^{*} \circ L \simeq f^{*} \circ t'^{*}\circ L$ we obtain an isomorphism $t'^{*}\circ L \simeq f^{*} \circ M \circ j$, which we still denote by $f^{*}\circ \alpha$. The pair $(f^{*}\circ M, f^{*}\circ \alpha)$ is an object of $(t'^{*}\circ L)(Q/P)$; there is a natural 
functor
   \[
   f^{*}\colon (t^{*}\circ L)(Q/P) \arr (t'^{*}\circ L)(Q/P)
   \]
(we leave it to the reader to define the action of $f^{*}$ on arrows).

This defines a pseudo-functor from $\catsch{X}$ into the 2-category of 
categories.

\begin{definition}
Let $X$ be a scheme, $j\colon P\arr Q$ a homomorphism of monoids, $L\colon P \arr \div X$ a \sm functor. We define the \emph{stack of roots} associated with these data, denoted by $(X, L)_{Q/P}$, or simply $X_{Q/P}$, as the fibered category over $\catsch{X}$ associated with the pseudo-functor above.
\end{definition}

Suppose that in the definition above $j\colon P \arr Q$ is a homomorphism of finitely generated monoids. From Proposition~\ref{prop:stack-of-charts} we see that the homomorphism $L\colon P \arr \div X$ corresponds to a morphism $X \arr \szp{P}$. Again from Proposition~\ref{prop:stack-of-charts} we obtain the following useful description of $X_{Q/P}$. The homomorphism $\ZZ[P] \arr \ZZ[Q]$ induced by $j$ induces a morphism $\sz{Q} \arr \sz{P}$. This is $\widehat{\jmath}\colon \widehat{Q} \arr \widehat{P}$ equivariant, where $\widehat{\jmath}$ is the homomorphism of algebraic groups over $\ZZ$ induced by $j$. This gives a morphism of algebraic stacks $\szp{Q} \arr \szp{P}$. This corresponds with the morphism
   \[
   \underhom(Q, \div_{\ZZ}) \arr \underhom(P, \div_{\ZZ})
   \]
induced by $j$. From Proposition~\ref{prop:stack-of-charts} we immediately obtain the following.

\begin{proposition}\label{prop:describe-root-stacks-locally}
The stack $X_{Q/P}$ is isomorphic to the fibered product
    \[
   X \times_{\szp{P}}\szp{Q}\,.
   \]
\end{proposition}

\begin{remark}\label{rmk:describe-root-stacks-2}
This can also be stated as follows. Let $L\colon P\rightarrow \div X$ be a \sm functor, corresponding, according to Proposition~\ref{prop:stack-of-charts}, to a morphism $X \arr \szp{P}$, i.e., to a $\widehat{P}$-torsor $\eta\colon E \arr X$ and a $\widehat{P}$-equivariant morphism $E \arr \sz{P}$. From the proof of Proposition~\ref{prop:stack-of-charts} we see that the $P\gp$-graded $\cO_{X}[P]$ algebra $A \eqdef \eta_{*}\cO_{E}$ is canonically isomorphic to $\bigoplus_{u \in P\gp}L\gp_{u}$. The functor $L\colon P \arr \div X$ sends $p \in P$ to the pair $(A_{\iota_{P}(p)}, x^{p})$.

The fibered product $X \times_{\szp{P}}\szp{Q}$ is the stack theoretic quotient
   \[
   [E\times_{\sz{P}}\sz{Q}/\widehat{Q}]\,,
   \]
where the action of $\widehat{Q}$ on the fibered product $E\times_{\sz{P}}\sz{Q} = \underspec_{X}(A \otimes_{\ZZ[P]} \ZZ[Q])$ is given by the natural action on the second factor, while on the first factor $\widehat{Q}$ acts through the natural homomorphism $\widehat{Q} \arr \widehat{P}$ induced by the embedding $P \subseteq Q$. In other words, $X \times_{\szp{P}}\szp{Q}$ is the relative spectrum of the $\cO_{X}$-algebra $A \otimes_{\ZZ[P]} \ZZ[Q]$, with the obvious grading.

This gives a description of the category of \qc sheaves on the stack $X_{Q/P}$ that will be used later. A \qc sheaf on $X_{Q/P}$ corresponds to a $\widehat{Q}$-equivariant \qc sheaf on $E\times_{\sz{P}}\sz{Q}$; and this corresponds to a $Q\gp$-graded \qc sheaf of modules over the sheaf of rings $B \eqdef A \otimes_{\ZZ[P]} \ZZ[Q]$.

Denote by $\pi\colon X_{Q/P} \arr X$ the projection. There is a tautological extension of the pullback \df structure $\pi^{P}$, which we will denote by $\Lambda\colon Q \arr \div X_{Q/P}$; if $v \in Q\gp$, then $\Lambda\gp\colon Q\gp \arr \pic X_{Q/P}$ sends $v$ into the sheaf $B[v]$ (by which we denote the sheaf $B$, but with the grading shifted by $v$, i.e., $B[v]_{v'} = B_{v+v'}$).

If $\pi\colon X_{Q/P} \arr X$ denotes the projection, the pushforward operation $\pi_{*}$ from \qc sheaves on $X_{Q/P}$ to \qc sheaves on $X$ corresponds to the operation that associates with each sheaf $N$ of $Q\gp$-graded \qc sheaf of modules over $B$ the part $N_{0}$ of degree $0$.
\end{remark}

\begin{corollary}
If $j\colon P \arr Q$ is a homomorphism of finitely generated monoids, $X$ is a scheme, and $L\colon P \arr \div X$ is a \sm functor, the stack $X_{Q/P}$ is algebraic and finitely presented over $X$.
\end{corollary}

Here is a variant of this definition. Suppose that $X$ is a scheme, $j\colon A \arr B$ a homomorphism of sheaves of monoids on $X\et$, $L\colon A \arr \div_{X\et}$ a morphism of \sm fibered categories. For each morphism of schemes $t\colon T \arr X$ we have a \sm functor $t^{*}L\colon t^{*}A \arr \div_{T\et}$, with which we associate a category $(t^{*}L)(t^{*}B/t^{*}A)$.

Suppose that $(M, \alpha)$ is an object of $(t^{*}L)(t^{*}B/t^{*}A)$, and that $f\colon T' \arr T$ is a morphism of $X$-schemes. The isomorphism $\alpha\colon t^{*}L \simeq M \circ t^{*}j$ pulls back to an isomorphism $f^{*}t^{*}L \simeq f^{*}(M \circ t^{*}j)$. By composing with the natural isomorphisms $f^{*}t^{*}L \simeq t'^{*}L$ and $f^{*}(M \circ t^{*}j) \simeq f^*{M} \circ f^{*}t^{*}j \simeq f^{*}M \circ t'^{*}j$ we obtain an isomorphism $f^{*}\alpha \colon t'^{*}L \simeq f^{*}M \circ t'^{*}j$, and thus an object $f^{*}(M, \alpha) \eqdef (f^{*}M, f^{*}\alpha)$ of $(t'^{*}L)(t'^{*}B/t'^{*}A)$. This construction extends naturally to a \sm functor
   \[
   f^{*}\colon (t^{*}L)(t^{*}B/t^{*}A) \arr (t'^{*}L)(t'^{*}B/t'^{*}A)
   \]
which in turn gives a pseudo-functor from $\catsch{X}$ to the 2-category of 
categories.

\begin{definition}
Let $X$ be a scheme, $j\colon A \arr B$ a homomorphism of sheaves of monoids on $X\et$, $L\colon A \arr \div_{X\et}$ a morphism of \sm fibered categories. We define the \emph{stack of roots} associated with these data, denoted by $(X, L)_{B/A}$, or simply $X_{B/A}$, as the fibered category over $\catsch{X}$ associated with the pseudo-functor above.
\end{definition}

\begin{remark}
Suppose that $X$ is a scheme, $L\colon A \arr \div_{X\et}$ a \df structure, $B/A$ a system of denominators. Let $t\colon T \arr X_{B/A}$ be a morphism, where $T$ is a scheme. Then the corresponding morphism $M\colon t^{*}B \arr \div_{T\et}$ is a \df structure (i.e., its kernel is trivial). This follows easily from the fact that $B$ is sharp.
\end{remark}

\begin{proposition}\label{prop:local-root-stack}
Let $(A,L)$ be a \df structure on $X$, $j\colon A \arr B$ a system of denominators,
   \[
   \xymatrix{
   P \ar[r]^{j_{0}}\ar[d]^{h} &Q \ar[d]^{k}\\
   A(X) \ar[r]^{j(X)} & B(X)
   }
   \]
a chart for $B/A$. Let $L_{0}\eqdef L(X) \circ h\colon P \arr \div X$. Then there is a canonical equivalence $(X,L_{0})_{Q/P} \simeq (X, L)_{B/A}$ of fibered categories over $\catsch{X}$.
\end{proposition}

\begin{proof}
It is easy to construct a cartesian functor $X_{B/A} \arr X_{Q/P}$. Let $t\colon T \arr X$ be a morphism of schemes; we denote by $t^{*}h\colon P \arr (t^{*}A)(T)$ the composite of $h\colon P \arr A(X)$ with the natural pullback homomorphism $t^{*}\colon A(X) \arr (t^{*}A)(T)$, and analogously for $t^{*}k\colon
Q\arr (t^{*}B)(T)$.

Suppose that $(M, \alpha)$ is an object of $X_{B/A}(T) = (t^{*}L)(t^{*}B/t^{*}A)$. Then $\alpha(T)$ is an isomorphism between the functors $t^{*}L(T)$ and $M(T) \circ t^{*}j(T)$, which go from $t^{*}A(T)$ to $\div T$; then $\alpha(T) \circ t^{*}h$ is an isomorphism from $t^{*}L(T) \circ t^{*}h$ to $M(T) \circ t^{*}j(T) \circ t^{*}h = M(T)\circ t^{*}k \circ j_{0}$; hence $(M(T)\circ t^{*} k, \alpha(T) \circ t^{*} h)$ is an object of $X_{Q/P}(T) = (t^{*} \circ L_{0})(Q/P)$. This construction extends naturally to a cartesian functor $X_{B/A} \arr X_{Q/P}$.

To go in the other direction, start from an object $(M_{0}, \alpha_{0})$ of $X_{Q/P}(T) = (t^{*} \circ L_{0})(Q/P)$. Denote by $K_{A}$ and $K_{B}$ the presheaf kernels of the morphisms $P_{T} \arr A$ and $Q_{T} \arr B$ induced by $t^{*}h$ and $t^{*}k$ respectively. From the characterization of Lemma~\ref{lem:char-kernel-chart}, it is easy to see that $K_{B}$ is also the kernel of the \sm functor $Q_{T} \arr \div_{T\et}$ induced by $M_{0}$. Consider the \df structure $M\colon t^{*}B = (Q_{T}/K_{B})\sh \arr \div_{T\et}$ induced by $M_{0}$ (Proposition~\ref{prop:chart->df}). Call $M_{A}$ the restriction of $M$ to $t^{*}A$, i.e., the composite $t^{*}A \xarr{t^{*}j} t^{*}B \xarr{M} \div_{T\et}$; then the composite $P \xarr{j_{0}} Q \xarr{t^{*}k} B(T) \xarr{M(T)} \div T$ factors through $M_{A}(T)\colon A(T) \arr \div T$. From the isomorphism $t^{*}\alpha_{0}\colon t^{*}L_{0} \simeq t^{*}
j\circ M_{0}$ and the restriction to 
$P(T)$ of the given isomorphism
$M(T)\circ t^{*}k\simeq M_0$, we obtain an isomorphism $\alpha\colon t^{*}L \simeq M \circ j$, because of the functoriality statement in Proposition~\ref{prop:chart->df}. The pair $(M, \alpha)$ is an object of $X_{B/A}(T)$.

We leave it to the reader to define a morphism of fibered categories $X_{Q/P} \arr X_{B/A}$ that associates $(M, \alpha)$ with $(M_{0}, \alpha_{0})$, and check that this yields a quasi-inverse to the morphism $X_{B/A} \arr X_{Q/P}$ defined above.
\end{proof}

\begin{proposition}
Let $(A,L)$ be a \df structure on $X$, $j\colon A \arr B$ a system of denominators. Then the fibered category $X_{B/A}$ is a finite and finitely presented algebraic stack over $X$. It is tame, in the sense of\/ \cite{AOV}. 

Furthermore, assume that for each geometric point $x\colon \spec\Omega \arr X$, the order of the quotient $B_{x}\gp/A_{x}\gp$ is prime to the characteristic of $\Omega$. Then $X_{B/A}$ is a Deligne--Mumford stack.
\end{proposition}

\begin{proof}
The fact that $X_{B/A}$ is a stack follows from standard arguments of descent theory, and is omitted.

To check the other conditions is a local question in the étale topology over $X$; hence we may assume that there is a chart
   \[
   \xymatrix{
   P \ar[d]\ar[r] & Q \ar[d]\\
   A(X) \ar[r]    & B(X)\,.
   }
   \]
Furthermore, if the order of the quotient $B_{x}\gp/A_{x}\gp$ is prime to the characteristic of $\Omega$ for each geometric point $x\colon \spec\Omega \arr X$, we may assume that the order of the finite group $Q\gp/P\gp$ is everywhere prime to the characteristic of each of the residue fields of $X$.

Call $G$ the kernel of the surjective homomorphism $\widehat{\jmath}\colon \widehat{Q} \arr \widehat{P}$ induced by $j$; it is a finite diagonalizable group, the Cartier dual of the finite group $Q\gp/P\gp$. It is smooth if the condition on the characteristic is verified. We have a cartesian diagram
    \[
   \xymatrix{
   [\spec \ZZ[Q]/G] \ar[r]\ar[d] &\szp{Q} \ar[d]\\
   \sz{P} \ar[r] &\szp{P}\,,
   }
   \]
which says that fppf locally on $X$ the stack $X_{Q/P}$ is a quotient by an action of $G$ over a scheme which is finite over $X$ (since $\ZZ[Q]$ is a finite extension of $\ZZ[P]$). This shows that it is finite and tame, and Deligne--Mumford if the conditions on the characteristic are verified.
\end{proof}

\section{Parabolic sheaves}\label{sec:parabolic}

\subsection{Categories of weights}

\begin{definition}
\label{def:weight-category}
	Given a monoid $A$, let $A\wt$ be the strict \sm category whose objects are elements of $A\gp$, and arrows $a\colon u\rightarrow v$ are elements $a$ of $A$ such that $u + \iota_{A}a = v \in A\gp$. The monoidal structure is given by the operations in $A\gp$ (for the objects) and $A$ (for the arrows).
\end{definition}

Notice that if $A$ is integral, $A\wt$ is a partially ordered set (that is, there is at most one arrow between any two objects of $A\wt$).

There is a natural \sm functor $A\rightarrow A\wt$, given at the level of objects by the function $\iota_{A}\colon A \arr A\gp$.

\begin{proposition}\label{prop:universal-to-pic}
Given a monoid $A$, a scheme $X$ and a \sm functor $L\colon A \arr \div X$, there exists a \sm functor $L\wt : A\wt \arr \pic X$, and a monoidal $2$-isomorphism $\Phi$
   \[
	\xymatrix@R=12pt{
	A \ar[d]_{L}	 \ar[r] & A\wt \ar[d]^{L\wt}\\
    \div X       \ar[r] \ar@{=>}[ur]^\Phi & \pic X }
   \]
such that for all $a \in A$ the diagram
	\[
	\xymatrix@R=12pt{
	& L_a \ar[r]^-{\Phi(a)} &       L\wt(a)\\
	\mathcal O_X\ar[ur]^{\sigma_a^L}\ar[dr]_{\epsilon^L}&&       \\
	& L_0  \ar[r]^-{\Phi(0)}&       
	L\wt(0)\ar[uu]_{L\wt(a)}}
	\]
 commutes.
 
Furthermore $L\wt$ and $\Phi$ are unique up to a unique isomorphism.
\end{proposition}

\begin{proof}

The uniqueness statement is easily proved, so we concentrate on constructing a solution $(L\wt,\Phi)$. 

First, consider the category $A\dwt$ whose objects are elements of $A\times A$ and whose arrows $c: (a,b)\rightarrow (a',b')$ are elements $c$ of $A$ such that $\iota_{A}(a+b'+c) = \iota_{A}(a'+b)$. There is a natural functor $A\dwt\rightarrow A\wt$ sending the object $(a,b)$ to the object $a-b$, and this is clearly a monoidal equivalence. The functor $A \arr A\wt$ factors as $A \arr A\dwt \arr A\wt$, where $A \arr A\dwt$ is defined by $a \arrto (a, 0)$; thus it is enough to produce a functor $L\dwt\colon A\dwt \arr \pic X$, together with an isomorphism of the composites $A \arr A\dwt \xarr{L\dwt} \pic X$ and $A \arr \div X \arr \pic X$, such that for each $a \in A$ the diagram
	\[
	\xymatrix@R=12pt{
	& L_a \ar[r]^-{\Phi(a)} &       L\dwt(a,0)\\
	\mathcal O_X\ar[ur]^{\sigma_a^L}\ar[dr]_{\epsilon^L}&&       \\
	& L_0  \ar[r]^-{\Phi(0)}&       
	L\dwt(0,0)\ar[uu]_{L\dwt(a)}}
	\]
commutes. Then we can define $L\wt$ by composing $L\dwt$ with a monoidal quasi-inverse $A\wt \arr A\dwt$ of the equivalence $A\dwt \arr A\wt$.

At the level of objects, we define $L\dwt$ by the obvious rule $L\dwt(a,b) = L_{a} \otimes L_{b}^{\vee}$. Given an arrow $c\colon (a, b) \arr (a', b')$, there is an element $d \in A$ such that $a + b' + c + d = a' + b + d$, we get isomorphisms 
   \begin{align*}
   L_{a} \otimes L_{b'} \otimes L_{c} \otimes L_{d}  & \simeq L_{a+b'+c+d}\\
   &= L_{a'+b+d}\\
   & \simeq L_{a'} \otimes L_{b} \otimes L_{d}\,,
   \end{align*}
hence an isomorphism
   \[
   L_{a} \otimes L_{b'} \otimes L_{c}
   \simeq
   L_{a'} \otimes L_{b},
   \]
which yields a homomorphism $L_{a} \otimes L_{b'} \arr L_{a'} \otimes L_{b}$, by sending a section $s$ of $L_{a} \otimes L_{b'}$ into the section of $L_{a'} \otimes L_{b}$ corresponding to $s \otimes \sigma_{c}^{L}$. This in turn yields a homomorphism
   \[
   L\dwt(c)
   \colon L_{a} \otimes L_{b}^{\vee} \arr L_{a'} \otimes L_{b'}^{\vee}.
   \]
The verification that the operation of $L\dwt$ on arrows preserves composition is long but straightforward. Hence we obtain the desired functor $L\dwt\colon A\dwt \arr \pic X$.

For each $a \in A$, the isomorphism $\Phi(a)\colon L_{a} \simeq L\dwt(a,0) \eqdef L_{a} \otimes L_{0}^{\vee}$ comes from the isomorphism $\epsilon^{L}\colon L_{0} \simeq \cO_{X}$. It is immediate to check that the diagram above is commutative. This completes the proof.
\end{proof}

\begin{remark}
Conversely, given a \sm functor $M\colon A\wt \arr \pic X$, any $a \in A$ defines an arrow $0 \arr a$ in $A\wt$; this yields an arrow $\cO_{X} \simeq M(0) \arr M(a)$ in $\pic X$. If we set $L_{a} = M(a)$ and call $s_{a}$ the section of $L_{a}$ corresponding to the morphism $\cO_{X} \arr L_{a}$ just defined, we obtain a monoidal functor $L\colon A \arr \div X$, such that $L\wt \simeq M$. In this way we obtain an equivalence of categories between \sm functors $A \arr \div X$ and \sm functors $A\wt \arr \pic X$.
\end{remark}

This construction generalizes to sheaves. Let $X$ be a scheme, $A$ be a sheaf of monoids on $X\et$. We denote by $A\wt \arr X\et$ the fibered category defined as follows. The objects are pair $(U, u)$, where $U \arr X$ is an étale morphism and $u\in A\gp(U)$. The arrows from $(U, u)$ to $(V, v)$ are pairs $(f, a)$, where $f\colon U \arr V$ is a morphism of $X$-schemes and $a$ is an element of $A(U)$ such that $u + \iota_{A}(a) = f^{*}v \in A\gp(U)$. Composition is defined by addition and pullback: if $(f, a)\colon (U, u) \arr (V,v)$ and $(g, b)\colon (V,v) \arr (W,w)$ are arrows, the composite is defined as
   \[
   (g,b) \circ (f, a) \eqdef (gf, f^{*}b + a).
   \]

\begin{proposition}\label{prop:universal-to-pic-stack}
	Given a sheaf of monoids $A$ on a scheme $X$ and a \sm functor $L\colon A \arr \div_{X\et}$, there exists a \sm functor $L\wt : A\wt \arr \pic_{X\et}$, and a monoidal cartesian $2$-isomorphism $\Phi$:

 \[
	\xymatrix@R=12pt{
	A \ar[d]_{L}	 \ar[r] & A\wt \ar[d]^{L\wt}\\
    \div_{X\et}      \ar[r] \ar@{=>}[ur]^\Phi & \pic_{X\et} }
\]
such that for all étale morphism $U \arr X$ and $a \in A(U)$ the following diagram commutes :

\[
\xymatrix@R=12pt{
& L_a \ar[r]^{\Phi(a)} &       L\wt(a)\\
\mathcal O_U\ar[ur]^{\sigma_a^L}\ar[dr]_{\epsilon^L_U}&&       \\
& L_0  \ar[r]_{\Phi(0)}&       L\wt(0)\ar[uu]_{L\wt(a)}}
\]

Furthermore $L\wt$ and $\Phi$ are unique up to a unique monoidal cartesian $2$\dash isomorphism.
	
\end{proposition}

\begin{proof}
To show the existence, we use Proposition \ref{prop:universal-to-pic} to produce,  for each étale morphism $U \arr X$, a solution $(L\wt(U),\Phi(U))$. The uniqueness statement in Proposition \ref{prop:universal-to-pic} shows that these solutions are compatible with restriction, hence define a global solution $(L\wt,\Phi)$.
\end{proof}

\begin{remark}
As before, we have an equivalence between the category of \sm functors $A \arr \div_{X}$ and \sm functors $A\wt \arr \pic_{X}$.
\end{remark}

\subsection{Parabolic sheaves}

Let $X$ be a scheme, $j\colon A \arr B$ a Kummer homomorphism of monoids, $L\colon A \arr \div X$ a \sm functor. We will always omit $j$ from the notation, and think of $A$ as a submonoid of $B$, and of $A\wt$ as a subcategory of $B\wt$.

Consider the extension $L\wt\colon A\wt \arr \pic X$ (Proposition~\ref{prop:universal-to-pic}); if $u \in A\wt$, for simplicity of notation we denote by $L_{u}$ the invertible sheaf $L\wt(u)$ image of $u$. If $u = \iota_{A}a$ for some $a \in A$, then $L\wt(u)$ is canonically isomorphic to $L_{a}$, and there should be no risk of confusion. Also, as usual when $a \in A$ we denote by $\sigma^{L}_{a}$ the corresponding section of $L_{a}$, so that $L(a) = (L_{a}, \sigma^{L}_{a})$.

If $u$ and $u'$ are in $A\wt$, we have a given isomorphism $L_{u+u'} \simeq L_{u}\otimes L_{u'}$, which we denote by $\mu^{L}_{u,v}$, or simply $\mu$, once again dropping the ${}\wt$ superscripts from the notation. Similarly, we denote by $\epsilon^{L}$, or $\epsilon$, the given isomorphism between $\cO_{X}$ and $L_{0}$.

\begin{definition}
	\label{def:para-sheaf}
A \emph{parabolic sheaf} $(E,\rho^E)$ on $(X, A, L)$ with denominators in $B/A$ consists of the following data.

\begin{enumeratea}

\item A functor $E\colon B\wt \arr \qcoh{X}$, denoted by $v \arr E_{v}$ at the level of objects, and by $b \arrto E_{b}$ at the level of arrows.

\item For any $u \in A\wt$ and $v \in B\wt$, an isomorphism of $\cO_{X}$-modules
   \[
   \rho^{E}_{u,v}\colon E_{u+v} \simeq L_{u} \otimes_{\cO_{X}} E_{v}\,,
   \]
which we will call \emph{pseudo-period isomorphisms}. If $a \in A$, we denote $\rho^{E}_{\iota_{A}a, v}$ by $\rho^{E}_{a,v}$.
\end{enumeratea}

These data are required to satisfy the following conditions. Let $u$, $u' \in A\wt$, $a \in A$, $b \in B$, $v \in B\wt$. Then the following diagrams commute.

\begin{enumeratei}

\item
   \[
   \xymatrix@C=50pt{
   E_{v} \ar[r]^-{E_{a}} \ar[d]^{\simeq}
      &E_{\iota_{A}a+v}\ar[d]^-{\rho^{E}_{a,v}}\\
   \cO_{X} \otimes E_{v} \ar[r]^-{\sigma^{L}_{a} \otimes \id_{E_{v}}}
      &L_{a}\otimes E_{v}
   }
   \]

\item
   \[
   \xymatrix @C = 40pt{
   E_{u+v} \ar[r]^-{\rho^{E}_{u,v}} \ar[d]_{E_{b}} 
   &L_{u} \otimes E_{v}
      \ar[d]^{\id_{L_{u}}\otimes E_{b}}\\
   E_{u+b+v} \ar[r]^-{\rho^{E}_{u,b+v}}  
   &L_{u} \otimes E_{b+v}
   }
   \]

\item
   \[
   \xymatrix@C=50pt{
   E_{u+u'+v} \ar[r]^-{\rho^{E}_{u+u', v}} \ar[d]^-{\rho^{E}_{u, u'+v}}
      &L_{u+u'} \otimes E_{v}
      \ar[d]^{\mu\otimes \id_{E_{v}}}\\
   L_{u} \otimes E_{u'+v} \ar[r]^-{\id_{L_{u}} \otimes \rho^{E}_{u',v}}
   & L_{u} \otimes L_{u'} \otimes E_{v}
   }
   \]

\item Finally, the composite
   \[
   \xymatrix@C=40pt{
   E_{v} = E_{0+v} \ar[r]^-{\rho^{E}_{0,v}}
   &L_{0} \otimes E_{v}\ar[r]^{\epsilon \otimes \id_{E_{v}}}
      &\cO_{X} \otimes E_{v}
   }
   \]
is the natural isomorphism $E_{v} \simeq \cO_{X} \otimes E_{v}$.

\end{enumeratei}

\end{definition}

\begin{remark}\label{rmk:cat-interpretation1}

This definition has the following high-level interpretation. There are natural functors ${+}\colon A\wt \times B\wt \arr B\wt$ (given by addition) and ${\otimes}\colon \pic X \times \qcoh X \arr \qcoh X$ (given by tensor product). These can be interpreted as action of the \sm categories $A\wt$ on $B\wt$ and of $\pic X$ on $\qcoh X$. Then the first two conditions mean that $\rho^{E}$ is an isomorphism of the composites $E\circ {+}$ and ${\otimes}\circ (L\wt \times E)$. The other two ensure that $E$ can be interpreted as an $A\wt$-equivariant functor.
	It is easy to check that the data of a parabolic sheaf on $(X, A, L)$ with denominators in  $B/A$ is equivalent to the data of a $A\wt$-morphism of modules categories $E\colon B\wt \rightarrow \qcoh X$ in the sense of \cite{Ost}, Definition 2.7.
\end{remark}

There is an abelian category $\qcoh_{X}(X, A, L)(Q/P)$ whose objects are quasi-coherent sheaves on $(X, A, L)$ with denominators in $Q/P$. An arrow $\Phi\colon E \arr E'$ is a natural transformation such that for all $u \in A\wt$ and $v \in B\wt$ the diagram
   \[
   \xymatrix @C=45pt{
   E_{u+v} \ar[r]^{\rho^{E}_{u,v}} \ar[d]^-{\Phi_{u+v}}
      & L_{u} \otimes E_{v} \ar[d]^{\id_{L_{u}} \otimes \Phi_{v}}\\
   E'_{u+v} \ar[r]^-{\rho^{E'}_{u,v}}& L_{u} \otimes E'_{v}
   }
   \]
commutes.

We will see that this category has tensor products and internal Homs.

There is also a sheafified version of the definition of parabolic sheaf.

\begin{definition}
	Let $X$ be a scheme, $(A, L)$ a coherent \df structure on $X$, $j\colon A \arr B$ a system of denominators.  A \emph{parabolic sheaf on $(X, A, L)$ with denominators in $B/A$} consists of the following data.

\begin{enumeratea}

\item A cartesian functor $E\colon B\wt \arr \qcoh_{X\et}$,  denoted by $v \arr E_{v}$ at the level of objects, and by $b \arrto E_{b}$ at the level of arrows.

\item For any $U \arr X$ in $X\et$, any $u \in A\wt(U)$ and $v \in B\wt(U)$, an isomorphism of $\cO_{U}$-modules
   \[
   \rho^{E}_{u,v}\colon E_{u+v} \simeq L_{u} \otimes E_{v}.
   \]

\end{enumeratea}

These data are required to satisfy the following conditions analogous to those of Definition~\ref{def:para-sheaf}, and the following.

If $f\colon U \arr V$ is an arrow in $X\et$, $u \in A\wt(V)$ and $v \in B\wt(V)$, then the isomorphism
   \[
   \rho^{E}_{f^{*}u,f^{*}v}\colon 
   E_{f^{*}(u+v)} = E_{f^{*}u + f^{*}v} \simeq L_{f^{*}u} \otimes E_{f^{*}v}
   \]
is the pullback of $j^{E}_{u, v}\colon E_{u+v} \simeq L_{u} \otimes E_{v}$.
\end{definition}

\begin{remark}\label{rmk:cat-interpretation2}

This definition can also be interpreted as in Remark~\ref{rmk:cat-interpretation2}, substituting categories with fibered categories.
\end{remark}

There is an abelian category $\qcoh(X, A, L)(B/A)$ whose objects are quasi-coherent sheaves on $(X, A, L)$ with denominators in $B/A$. A homomorphism of parabolic sheaves is defined as in the case when $A$ and $B$ are fixed monoids.

\begin{proposition}\label{prop:local-parabolic}
Let $(A,L)$ be a \df structure on $X$, $j\colon A \arr B$ a system of denominators,
   \[
   \xymatrix{
   P \ar[r]^{j_{0}}\ar[d]^{h} &Q \ar[d]^{i}\\
   A(X) \ar[r]^{j(X)} & B(X)
   }
   \]
a chart for $B/A$. Let $L_{0}\eqdef L(X) \circ h\colon P \arr \div X$. Then there is a canonical equivalence of abelian categories of $\qcoh(X, A, L)(B/A)$ with $\qcoh(X, P, L_{0})(Q/P)$.
\end{proposition}

\begin{proof}
We begin by the obvious definition of the equivalence: at the level of objects, if $(E,\rho^E)$ is a parabolic sheaf on $(X, A, L)$ with denominators in $B/A$, we can associate with it a parabolic sheaf with denominators in $Q/P$: $(E(X)\circ i,\rho^E(X)\circ(h\times i))$, and it is also clear how to define the functor at the level of morphisms. So we get a functor $\qcoh(X, A, L)(B/A) \rightarrow \qcoh(X, P, L_{0})(Q/P)$, and it is easy to check that it is fully faithful. So we now prove that the functor is in fact essentially surjective. 

Let $(E_0,\rho^{E_0})$ be a parabolic sheaf with denominators in $Q/P$ with respect to $L_{0}$. We must prove that the cartesian
functor $(E_0)_X:Q_X\wt\rightarrow \qcoh_{X\et}$ associated with $E_0$ factors trough $B\wt$, and an analogous statement for $\rho^{E_0}$. As in Definition \ref{lem:char-kernel-chart}, let $K_A$ (respectively $K_B$) the kernel of the morphism $P_X\rightarrow A$ (respectively $Q_X\rightarrow B$) induced by $h$ (respectively by $i$). Let $B\pre=Q_{X}/K_{B}$ the quotient presheaf, since $B\wt$ is the stackification of $(B\pre)\wt$, and $\qcoh_{X\et}$ is a stack, it is enough to show that $(E_0)_X$ factors trough $(B\pre)\wt$.

Let us first describe the cartesian category $(B\pre)\wt$ on $X\et$. Above an \'etale morphism $U\rightarrow X$, its objects are by definition elements of
   \[
   (B\pre)(U)\gp=\left(Q_{X}(U)/K_{B}(U)\right)\gp,
   \]
that is, equivalence classes $\cl(u)$ of elements of $u$ in $Q_{X}(U)\gp$ for the equivalence relation $u\sim v$ when there exists elements $k,l$ in $K_B(U)$ such that $u+\iota_Q(k)=v+\iota_Q(l)$, where $\iota_Q$ denotes as usual the morphism $Q_X\rightarrow Q_X\gp$. Maps from $\cl(u)$ to $\cl(v)$ are given by elements $\cl(q)$ of $(B\pre)(U)=Q_{X}(U)/K_{B}(U)$ such that $\cl(u)+\iota(\cl(q))=\cl(v)$.

We now introduce a category $(B\pre)\dwt$ that is a quotient of the Gabriel-Zisman localization of $Q_X\wt$ with respect to maps in $K_B$, such that it is equivalent to $(B\pre)\wt$. Above an \'etale morphism $U\rightarrow X$ objects of $(B\pre)\dwt(U)$ are elements of $Q_{X}(U)\gp$, and maps from $u$ to $v$ are equivalence classes of pairs $\cl((q,k))$, where $q\in Q_X(U)$ and $k\in K_B(U)$ are such that $u+\iota_Q(q)=v+\iota_Q(k)$, for the equivalence relation $(q,k)\sim(q',k')$ if there exists $k_1,k_2\in K_B(U)$ such that $q+k_1=q'+k_2$. 

The cartesian functor $Q_X\wt\rightarrow (B\pre)\wt$ sending an object $u$ to $\cl(u)$ and a morphism $q$ to $\cl(q)$ factors trough a cartesian functor $Q_X\wt\rightarrow (B\pre)\dwt$ sending an object $u$ to itself and a morphism $q$ to $\cl((q,0))$, and trough a cartesian functor $(B\pre)\dwt\rightarrow (B\pre)\wt$ sending an object $u$ to $\cl(u)$ and a morphism $\cl((q,k))$ to $\cl(q)$. One checks immediately that this last functor is well defined and a cartesian equivalence, so this is enough to produce a factorization of $(E_0)_X$ trough $(B\pre)\dwt$. To achieve this, we need the following lemma:

\begin{lemma} Let $U\rightarrow X$ be an \'etale morphism.
	For any $v\in Q_X\gp(U)$ and $k\in K_B(U)$ the morphism $((E_0)_X)_k:((E_0)_X)_v\rightarrow ((E_0)_X)_{\iota_Q(k)+v}$ is an isomorphism.
\end{lemma}

\begin{proof}
	This is a local problem in the \'etale topology, hence by Lemma \ref{lem:char-kernel-chart}, we can assume that there exists a positive integer $m$ and an element $l\in K_A(U)$ such that $mk=l$. Definition \ref{def:para-sheaf} ensures that the diagram:

	   \[
   \xymatrix@C=50pt{
   ((E_0)_X)_{v} \ar[r]^-{((E_0)_X)_{l}} \ar[d]^{\simeq}
      &((E_0)_X)_{\iota_{Q}l+v}\ar[d]^-{\rho^{(E_0)_X}_{l,v}}\\
      \cO_{U} \otimes ((E_0)_X)_v \ar[r]^-{\sigma^{(L_0)_X}_{l} \otimes \id}
      &L_{l}\otimes ((E_0)_X)_v
   }
   \]

   is commutative. Since $l\in K_A(U)$, the morphism  $\sigma^{(L_0)_X}_{l}$ is invertible, moreover the fact that $mk=l$ and the functoriality of $E_0$ show that $((E_0)_X)_l:((E_0)_X)_v\rightarrow ((E_0)_X)_{\iota_Q(l)+v}$ factors trough $((E_0)_X)_k:((E_0)_X)_v\rightarrow ((E_0)_X)_{\iota_Q(k)+v}$, hence this last morphism is left invertible. A similar argument shows that $((E_0)_X)_k$ is right invertible.
\end{proof}
Thanks to the lemma, we can define a cartesian functor $(B\pre)\dwt\rightarrow \qcoh_{X\et}$ that above an \'etale morphism $U\rightarrow X$ sends the arrow $\cl((q,k)): u\rightarrow v$ to the composite of $((E_0)_X)_q:((E_0)_X)_u\rightarrow ((E_0)_X)_{\iota_Q(k)+v}$ with the inverse of $((E_0)_X)_k:((E_0)_X)_v\rightarrow ((E_0)_X)_{\iota_Q(k)+v}$. The functoriality of $E_0$ shows that this is well defined, and produces a factorization of $(E_0)_X$ trough $(B\pre)\dwt$, hence a factorization of $(E_0)_X$ trough $(B\pre)\wt$.

The proof that $\rho^{E_0}$ does also factor trough $A\wt\times B\wt$ is similar, so we omit it. Hence the functor $\qcoh(X, A, L)(B/A) \rightarrow \qcoh(X, P, L_{0})(Q/P)$ we have defined is essentially surjective, and so this is an equivalence.\end{proof}

\subsection{Internal Hom and tensor product}

Let $E,E'$ be two objects of the category $\qcoh(X, P, L_{0})(Q/P)$. First, we define a quasi-coherent sheaf $\bHom(E,E')_0$  on $X$ by the usual rule: for every étale map $U\rightarrow X$, \[\bHom(E,E')_0(U)=\Hom_U(E_{|U},E'_{|U})\]
where $\Hom_U$ is the $\mathcal O(U)$-module of all homomorphisms of parabolic sheaves on $U$ defined in the previous paragraph.

If instead we start from an object $G$ of the category $\qcoh(X)$ and an object $E$ of $\qcoh(X, P, L_{0})(Q/P)$, we can consider the external tensor product $G\otimes E$ as the object of $\qcoh(X, P, L_{0})(Q/P)$ given on objects by the rule: for $v\in Q\gp$,  $(G\otimes E)_v=G\otimes E_v$.

These two operations are related by the formula 
\[\bHom(G\otimes E,E')_0\simeq \bHom(G,\bHom(E,E')_0)\]
 where the second $\bHom$ is the usual internal Hom in $\qcoh(X)$.

Now for $v\in Q\gp$ and $E$ an object of $\qcoh(X, P, L_{0})(Q/P)$, we can define the twist $E[v]$ by the rule: for $v'\in Q\gp$, $E[v]_{v'}=E_{v+v'}$.

 If we start again from two objects $E,E'$ of $\qcoh(X, P, L_{0})(Q/P)$, we have for $u\in P\gp$ and $v\in Q\gp$ canonical isomorphisms 
\begin{align*}
	\bHom(E,E'[u+v])_0 &\simeq \bHom(E[-u],E'[v])_0\\
	&\simeq \bHom(L_{-u}\otimes E,E'[v])_0\\
&\simeq \bHom(L_{-u},\bHom(E,E'[v])_0)\\
&\simeq L_{u}\otimes\bHom(E,E'[v])_0\\
	\end{align*}
This shows that the functor $v\rightarrow \bHom(E,E'[v])_0$ can be endowed with a structure of a parabolic sheaf, denoted by $\bHom(E,E')$. Thus we have an internal Hom in $\qcoh(X, P, L_{0})(Q/P)$, and the definition of the tensor product follows from the standard formula:
\[  \bHom(E\otimes E', E'')\simeq \bHom(E, \bHom(E',E''))\]

Along the same lines, we also can define an internal Hom in $\qcoh(X, A, L)(B/A)$, the only difference being that we can twist only locally. Thus for two objects $E,E'$ of $\qcoh(X, A, L)(B/A)$, we define for $U\rightarrow X$ étale and $v\in B\gp(U)$:
\[\bHom(E,E')_v=\bHom(E_{|U},E'_{|U}[v])_0\]

The tensor product is defined by the formula above.

\section{The main theorem}\label{sec:main-theorem}

In this section we will use the notion of a \df structure on an algebraic stack, which is the immediate generalization of the notion of \df structure on a scheme.

The following is the main result of this paper.

\begin{theorem}\label{thm:main}
Let $(A,L)$ be a coherent \df structure on a scheme $X$ and let $B/A$ be a system of denominators. Then there is a canonical tensor equivalence of abelian categories between the category $\qcoh(X,A,L)(B/A)$ of parabolic sheaves on $(X, A, L)$ with denominator in $B$ and the category $\qcoh X_{B/A}$ of \qc sheaves on the stack $(X,L)_{B/A}$.
\end{theorem}

\begin{proof}
Let us construct a functor $\Phi\colon \qcoh X_{B/A} \arr \qcoh(X,A,L)(B/A)$. Denote by $\pi\colon X_{B/A} \arr X$ the projection. On the stack $X_{B/A}$ we have a canonical \df structure $\Lambda\colon \pi^{*}B \arr \div_{X_{B/A}}$, with an isomorphism of the restriction of $\Lambda$ to $A$ with the pullback of $L$ to $X_{B/A}$. Consider the functor $\Lambda\wt\colon B\wt \arr \pic_{X_{B/A}}$; for each étale morphism $U \arr X$ and each $v \in B\wt(U)$, we set $\Lambda_{v} \eqdef\Lambda\wt(v)$. If $u \in A(U)$, then $\Lambda_{u}$ is canonically isomorphic to $\pi^{*}L_{u}$, where $L_{u} \eqdef L\wt(u)$.

Let $F$ be a \qc sheaf on $X_{B/A}$. We need to define a parabolic sheaf $\Phi F$ on $(X,A,L)$ with coefficients in $B$. For each étale morphism $U\arr X$ and each $v \in B\wt(U)$, we set
   \[
   (\Phi F)_{v} \eqdef \pi_{*}(F \otimes_{\cO_{X_{B/A}}} \Lambda_{v}).
   \]
If $b \in B(U)$ and $v \in B\wt(U)$, the homomorphism $(\Phi F)_{b}\colon (\Phi F)_{v} \arr (\Phi F)_{b+v}$ is induced via $\pi_{*}$ by the homomorphism 
   \[
   F\otimes \Lambda_{v} \simeq F\otimes \Lambda_{v} \otimes \cO
   \xarr{\id_{F\otimes \Lambda_{v}}\otimes \sigma^{\Lambda}_{b}}
   F \otimes \Lambda_{v} \otimes \Lambda_{b}
   \simeq F \otimes \Lambda_{b + v}.
   \]
Given $u \in A(U)$ and $v \in B(U)$, the isomorphism
   \[
   \rho^{\Phi F}_{u,v}\colon (\Phi F)_{u+v} \simeq L_{u} \otimes (\Phi F)_{v}
   \]
is obtained via the following sequence of isomorphisms, using the projection formula for the morphism $\pi$:
   \begin{align*}
   (\Phi F)_{u+v} &=
   \pi_{*}(F \otimes_{\cO_{X_{B/A}}} \Lambda_{u+v})\\
   & \simeq
   \pi_{*}(F \otimes_{\cO_{X_{B/A}}} \Lambda_{v} \otimes \Lambda_{u})\\
      & \simeq
   \pi_{*}(F \otimes_{\cO_{X_{B/A}}} \Lambda_{v} \otimes \pi^{*}L_{u})\\
      & \simeq
   L_{u} \otimes \pi_{*}(F \otimes_{\cO_{X_{B/A}}} \Lambda_{v})\\
         & =
   L_{u} \otimes (\Phi F)_{v}\,.
   \end{align*}
We leave it to the reader to show that the pair $(\Phi F, \rho^{\Phi F})$ is a parabolic sheaf. This function on objects extends to an additive functor
   \[
   \Phi\colon \qcoh X_{B/A} \arr \qcoh(X,A,L)(B/A)
   \]
in the obvious way.

We claim that $\Phi$ is an equivalence. This is a local problem in the étale topology: this can be proved as follows.

First all, if $U \arr X$ is an étale morphism, denote by $U_{B/A}$ the stack of roots of the restriction $L_{U}$ of $L$ to $U\et$ with respect to the restriction $B_{U}$ of $B$. Then there are fibered categories $\qcoh_{X_{B/A}}$ and $\qcoh_{(X,A,L)(B/A)}$, whose fiber categories over an étale morphism $U \arr X$ are $\qcoh U_{B/A}$ and $\qcoh(U,A_{U},L_{U})(B_{U}/A_{U})$ respectively. The functor $\Phi$ defined above extends to a cartesian functor $\qcoh_{X_{B/A}} \arr \qcoh_{(X,A,L)(B/A)}$. Now, the point is that both $\qcoh_{X_{B/A}}$ and $\qcoh_{(X,A,L)(B/A)}$ are stacks in the étale topology. This follows from straightforward but lengthy descent theory arguments, which we omit. Of course, to check that a cartesian functor between stacks is an equivalence is a local problem.

So we may assume that there exists a chart
   \[
   \xymatrix{
   P \ar[r]^{j_{0}}\ar[d]^{h} &Q \ar[d]^{k}\\
   A(X) \ar[r]^{j(X)} & B(X)
   }
   \]
for $B/A$. Set $L_0= h \circ L(X)$; according to Proposition~\ref{prop:local-root-stack} and Proposition~\ref{prop:local-parabolic}, there are equivalences between the categories $\qcoh X_{B/A}$ and $\qcoh X_{Q/P}$, and between $\qcoh(X,A,L)(B/A)$ and $\qcoh(X,P, L_0)(Q/P)$. 

The functor $\qcoh X_{Q/P} \arr \qcoh(X,P,L_0)(Q/P)$, which we still denote by $\Phi$, is described as follows. We still denote by $\pi\colon X_{Q/P} \arr X$ the projection. For each $p \in P$ we denote by $L_{p}$ the invertible sheaf $L_{h(p)}$ on $X$; analogously, if $q \in Q$ we denote by $\Lambda_{q}$ the invertible sheaf on $X_{Q/P}$ corresponding to $\Lambda_{k(q)}$ on $X_{B/A}$. The functor $\Phi\colon \qcoh X_{Q/P} \arr \qcoh(X,P,L_0)(Q/P)$ sends  a \qc sheaf $F$ on $X_{Q/P}$ into $(\Phi F, \rho^{\Phi F})$, where $(\Phi F)_{p} \eqdef \pi_{*}(F \otimes L_{p})$, and $\rho^{\Phi F}$ is defined as above.  We need to check that this functor $\Phi$ is an equivalence.

We use the description of Proposition~\ref{prop:stack-of-charts}. The functor $L_0\colon P \arr \div X$ corresponds to a morphism $X \arr \szp{P}$, i.e., to a $P\gp$-torsor $\eta\colon T \arr X$ with an equivariant morphism $T \arr \sz{P}$. Denote by $A \eqdef \eta_{*}\cO_{T}$ the associated sheaf of $P\gp$-graded $\cO_{X}[P]$-algebras. According to Remark~\ref{rmk:describe-root-stacks-2}, the category $\qcoh X_{Q/P}$ is equivalent to the category of sheaves of $Q\gp$-graded $A\otimes_{\ZZ[P]} \ZZ[Q]$-modules. Set $B \eqdef A\otimes_{\ZZ[P]} \ZZ[Q]$. The functor $\pi$ from \qc sheaves on $X_{Q/P}$ to \qc sheaves on $X$ sends such a sheaf $F$ into the part $F_{0}$ of degree~$0$. Since for each $v \in Q\gp$, the sheaf $\Lambda_{v}$ corresponds to the shifted sheaf $B[v]$, the sheaf $\pi_{*}(\Lambda_{v}\otimes F)$ will be the part $F_{v}$ of degree $v$. 

Then the functor $\Phi$ is interpreted to the functor that sends such a sheaf $F$ of $Q\gp$-graded $A\otimes_{\ZZ[P]} \ZZ[Q]$-modules into the parabolic sheaf $\Phi F\colon Q\wt \arr \qcoh X$ sending $v\in Q\wt$ to $F_{v}$. If $q \in Q$ and $\iota_{Q}q + v = v'$, so that $q$ gives an arrow in the category $Q\wt$, the image $(\Phi F)_{q}\colon F_{v} \arr F_{v'}$ is given by multiplication by $x^{q}$. 

Now, let $u \in P\wt$ and $v \in Q\wt$. The sheaf $L_{u}$ on $X$ is isomorphic to the sheaf $A_{u}$; multiplication gives an isomorphism $A_{u}\otimes_{\cO_{X}}F_{v} \arr F_{u+v}$. Then $\rho^{\Phi F}\colon F_{u+v} \simeq A_{u}\otimes_{\cO_{X}}F_{v}$ is its inverse.

With this description, $\Phi F\colon \qcoh X_{Q/P} \arr \qcoh(X, P, L_0)(Q/P)$ is very easily seen to be an isomorphism. Let us construct a quasi-inverse
   \[
   \Psi\colon \qcoh(X, P, L_0)(Q/P) \arr \qcoh X_{Q/P}\,.
   \]
If $(E, \rho^{E})$ is a parabolic sheaf, we define the \qc sheaf $\Psi E$ on $X$ as the direct sum $\bigoplus_{v \in Q\gp} E_{v}$. The sheaf $\Psi E$ is in fact an $A$-module: since $A = \bigoplus_{u \in P\gp}L_{u}$, we define the homomorphism
   \[
   A \otimes_{\cO_{X}}\Psi E =
   \bigoplus_{\substack{u \in P\gp\\v \in Q\gp}} L_{u} \otimes E_{v}
   \arr \bigoplus_{v \in Q\gp} E_{v}
   \]
via the isomorphisms $(\rho^{E}_{u,v})^{-1}\colon L_{u} \otimes E_{v} \simeq E_{u+v}$. The sheaf $\Psi E$ is also a sheaf of $\ZZ[Q]$-algebras: for each $q \in Q$, we let $x^{q}$ act on $\Psi E$ by sending $E_{v}$ into $E_{\iota_{Q}(q) + v}$ as the homomorphism $E_{q}$. Thus, $\ZZ[P]$ acts on $\Psi E$ in two ways, by the embedding $\ZZ[P] \subseteq \ZZ[Q]$ and via the morphism to $A$ coming from the structure of $A$ as a sheaf of $\cO_{X}[P]$-algebra. Condition~(i) in the definition of a parabolic sheaf (Definition~\ref{def:para-sheaf}) ensures that these two actions coincides, and so gives $\Psi E$ the structure of a $Q\gp$-graded $A \otimes_{\ZZ[P]} \ZZ[Q]$-module, corresponding to a \qc sheaf on $X_{Q/P}$. 

We leave it to the reader to define the action of $\Psi$ on arrows, and show that it gives a quasi-inverse to $\Phi$.

It remains to prove that $\Phi$ is compatible with tensor products. It is enough to show that given $F,F'$ \qc sheaves on $X_{B/A}$, there is  a natural isomorphism: \[ \Phi\bHom(F,F')\simeq \bHom (\Phi(F),\Phi(F'))\]
Let $U\rightarrow X$ be an étale map and $v\in B(U)\gp$. On one hand we have 
\[\Phi\bHom(F,F')_v\simeq \pi_* \bHom(F,F'\otimes \Lambda_v)\]
and on the other hand 
\[\bHom (\Phi(F),\Phi(F'))_v\simeq \bHom(\phi(F)_{|U},\phi(F'\otimes \Lambda_v)_{|U})_0\]
but the equivalence of categories we have just proven shows that these sheaves are canonically isomorphic.
\end{proof}

\begin{example}
Let $(L_{1}, s_{1})$, \dots, $(L_{r}, s_{r})$ be invertible sheaves with sections on a scheme $X$; let $L\colon A \arr \div_{X\et}$ the \df structure that they generate (Definition~\ref{def:generated-df}). By definition, this $L$ has a chart $\NN^{r} \arr A(X)$. Let $d_{1}$, \dots,~$d_{r}$ be positive integers and $Q$ be the monoid $\frac{1}{d_{1}}\NN \times \dots \times\frac{1}{d_{r}}\NN$, with the natural embedding $\NN^{r} \subseteq Q$. The stack $(X, L)_{Q/\NN^{r}}$ is the fibered product
   \[
   \sqrt[\leftroot{-2}\uproot{2}d_{1}]{(L_{1}, s_{1})} \times_{X} \dots
   \times_{X} \sqrt[\leftroot{-2}\uproot{2}d_{r}]{(L_{r}, s_{r})}
   \]
of root stacks (in the sense of \cite[Appendix B]{dan-tom-angelo2008}). Thus we reproved and generalized the correspondence between parabolic sheaves and sheaves on root stacks of \cite{borne-para1} and \cite{borne-para2}.
\end{example}

\bibliography{paralog}
\end{document}